\documentclass[11pt]{article}
\usepackage{amssymb,url,amsmath,graphicx, amsthm}
\usepackage{accents}
\usepackage{tikz}
\usetikzlibrary{positioning, calc, chains}
\usetikzlibrary{shapes.misc}
\usetikzlibrary{shapes.symbols}
\usetikzlibrary{shapes.geometric}
\usetikzlibrary{shapes.arrows}
\usetikzlibrary{fit}
\usetikzlibrary{shadows}

\usepackage{subcaption}

\newtheorem{theorem}{Theorem}[section]
\newtheorem{definition}{Definition}[section]
\newtheorem{proposition}{Proposition}[section]
\newtheorem{lemma}{Lemma}[section]

\newtheorem{example}{Example}[section]

\newcommand{\ie}{\textit{i}.\textit{e}.}

\newcommand{\citep}[1]{\cite{#1}}

\title{A general approach to transforming finite elements}
\date{}
\author{Robert C. Kirby\thanks{Department of Mathematics, Baylor
    University; One Bear Place \#97328; Waco, TX 76798-7328.
    Email: robert\_kirby@baylor.edu.  This
    work was supported by NSF grant 1525697.}}
\begin{document}
\maketitle

\begin{abstract}
  The use of a \emph{reference element} on which a finite element
  basis is constructed once and mapped to each cell in a mesh greatly
  expedites the structure and efficiency of finite element codes.
  However, many famous finite elements such as Hermite, Morley,
  Argyris, and Bell, do not possess the kind of
  equivalence needed to work with a reference element in the standard
  way.  This paper gives a generalizated approach to mapping bases for
  such finite elements by means of studying relationships between the
  finite element nodes under push-forward.  {\bf MSC 2010}: 65N30.
  \emph{Keywords}: Finite element method, basis function, pull-back. 
\end{abstract}

\section{Introduction}
At the heart of any finite element implementation lies the evaluation
of basis functions and their derivatives on each cell in a mesh.
These values are used to compute local integral
contributions to stiffness matrices and load vectors, which are
assembled into a sparse matrix and then passed on to an
algebraic solver.  While it is fairly
easy to parametrize local integration routines over basis
functions, one must also provide an implementation of those basis
functions.  Frequently, finite element codes use a
\emph{reference element}, on which a set of basis functions is constructed
once and mapped via coordinate change to each cell in
a mesh.  Alternately, many finite element bases can be expressed in
terms of barycentric coordinates, in which case one must simply
convert between the physical and barycentric coordinates on each cell
in order evaluate basis functions.  Although we refer the reader to
recent results on \emph{Bernstein
  polynomials}~\citep{ainsworth2011bernstein, kirby2011fast} for
interesting algorithms in the latter case, the prevelance of the
reference element paradigm in modern high-level finite element
software~\citep{bangerth_deal_2007,intrepid,LoggMardalEtAl2012a,long2010unified,prud2012feel,rathgeber2016firedrake} we
shall restrict ourselves to the former.

The development of FIAT~\citep{Kir04} has had a significant impact on
finite element software, especially through its adoption in high-level
software projects such as FEniCS~\citep{LoggMardalEtAl2012a} and
Firedrake~\citep{rathgeber2016firedrake}. 
FIAT provides tools to describe and construct reference bases
for arbitrary-order instances of many common and unusual finite
elements.  Composed with a domain-specific language for variational
problems like UFL~\citep{AlnaesEtAl2012} and a form compiler mapping UFL into
efficient code for element integrals~\citep{tsfc,KirbyLogg2006a,luporini2014coffee} gives a powerful, user-friendly tool chain. 

However, any code based on the reference element paradigm operates
under the assumption that finite elements satisfy a certain
kind of \emph{equivalence}.  Essentially, one must have a
pull-back operation that puts basis functions on each cell
into one-to-one correspondence with the reference basis functions.
Hence, the original form of 
\texttt{ffc}~\citep{KirbyLogg2006a} used only (arbitrary order) Lagrange finite
elements, although this was generalized to $H(\mathrm{div})$ and
$H(\mathrm{curl})$ elements using Piola transforms
in~\citep{RognesKirbyEtAl2009a}. Current technology
captures the full simplicial discrete de Rham complex and
certain other elements, but many famous elements are not included.
Although it is possible to construct reference elements in FIAT or some
other way, current form compilers or other high-level libraries do
not provide correct code for mapping them.

\begin{figure}
  \begin{subfigure}[t]{0.2\textwidth}
    \begin{center}
    \begin{tikzpicture}[scale=2.0] 
    \draw[fill=yellow] (0,0) -- (1, 0) -- (0, 1) -- cycle;
    \foreach \i in {0, 1, 2, 3} {
      \foreach \j in {0,...,\i}{
      \draw[fill=black] (1-\i/3, \j/3) circle (0.02);
      }
      }
    \end{tikzpicture}
    \end{center}
  \caption{Cubic Lagrange}
  \label{lag3}
  \end{subfigure}
  \hfill
  \begin{subfigure}[t]{0.18\textwidth}
    \begin{center}
  \begin{tikzpicture}[scale=2.0] 
    \draw[fill=cyan] (0,0) -- (1, 0) -- (0, 1) -- cycle;
    \foreach \i/\j in {0/0, 1/0, 0/1}{
      \draw[fill=black] (\i, \j) circle (0.02);
      \draw (\i, \j) circle (0.05);
    }
    \draw[fill=black] (1/3, 1/3) circle (0.02);
  \end{tikzpicture}
  \end{center}
  \caption{Cubic Hermite}
  \label{herm3}
  \end{subfigure}
  \hfill
  \begin{subfigure}[t]{0.18\textwidth} 
    \begin{center}
    \begin{tikzpicture}[scale=2.0]
    \draw[fill=green] (0,0) -- (1, 0) -- (0, 1) -- cycle;
    \foreach \i/\j in {0/0, 1/0, 0/1}{
      \draw[fill=black] (\i, \j) circle (0.02);
    }
    \foreach \i/\j/\n/\t in {0.5/0.0/0.0/-1, 0.5/0.5/0.707/0.707, 0.0/0.5/-1/0}{
      \draw[->] (\i, \j) -- (\i+\n/10, \j+\t/10);
    }
    \end{tikzpicture}
    \end{center}
      \caption{Morley}
  \label{morley}
    \end{subfigure}
  \begin{subfigure}[t]{0.18\textwidth}
    \begin{center}
  \begin{tikzpicture}[scale=2.0] 
    \draw[fill=gray] (0,0) -- (1, 0) -- (0, 1) -- cycle;
    \foreach \i/\j in {0/0, 1/0, 0/1}{
      \draw[fill=black] (\i, \j) circle (0.02);
      \draw (\i, \j) circle (0.05);
      \draw (\i, \j) circle (0.08);
    }
    \foreach \i/\j/\n/\t in {0.5/0.0/0.0/-1, 0.5/0.5/0.707/0.707, 0.0/0.5/-1/0}{
      \draw[->] (\i, \j) -- (\i+\n/10, \j+\t/10);
      }
  \end{tikzpicture}
  \end{center}
  \caption{Quintic Argyris}
  \label{arg}
  \end{subfigure}
  \hfill
  \begin{subfigure}[t]{0.18\textwidth}
    \begin{center}
  \begin{tikzpicture}[scale=2.0] 
    \draw[fill=pink] (0,0) -- (1, 0) -- (0, 1) -- cycle;
    \foreach \i/\j in {0/0, 1/0, 0/1}{
      \draw[fill=black] (\i, \j) circle (0.02);
      \draw (\i, \j) circle (0.05);
      \draw (\i, \j) circle (0.08);
    }
  \end{tikzpicture}
  \end{center}
  \caption{Bell}
  \label{bell}
  \end{subfigure}
  \caption{Some famous triangular elements.  Solid dots represent
    point value degrees of freedom, smaller circles represent
    gradients, and larger circles represent the collection of second
    derivatives.  The arrows indicate directional derivatives
    evaluated at the tail of the arrow.}
    \label{fig:els}
\end{figure}
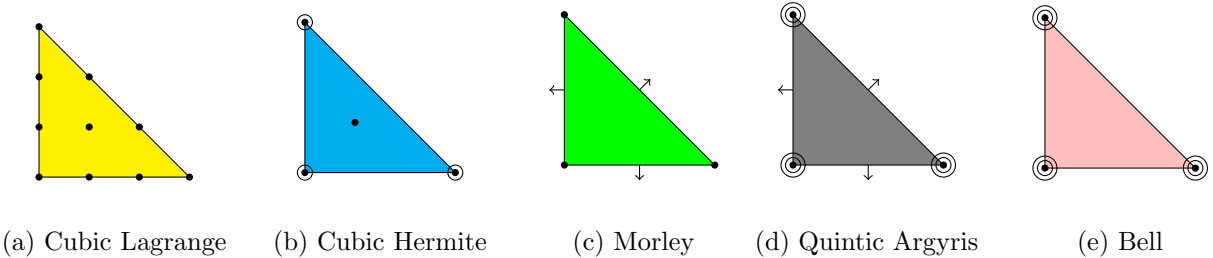

Elements such as Hermite~\citep{ciarlet1972general},
Argyris~\citep{argyris1968tuba}, Morley~\citep{morley1971constant},
and Bell~\citep{bell1969refined}, shown alongside the Lagrange element in Figure~\ref{fig:els}, do not satisfy the proper
equivalence properties to give a simple relationship between the
reference basis and nodal basis on a general cell.  Typically,
implementations of such elements require special-purpose code for
constructing the basis functions separately on each element, which can
cost nearly as much in terms of work and storage as building the element
stiffness matrix itself.  It also requires a different internal
workflow in the code.  Although Dom\'\i nguez and
Sayas~\citep{dominguez2008algorithm} give a technique for mapping 
bases for the Argyris element and a separate computer
implementation is available
(\url{https://github.com/VT-ICAM/ArgyrisPack}) and Jardin~\cite{jardin2004triangular} gives a per-element construction technique for the Bell element, these represents the exception rather than the rule.  The literature contains no general
approach for constructing and mapping finite element bases in the
absence of affine equivalence or a suitable generalization thereof.

In this paper we provide such a general theory for transforming
finite elements that supplements the theory on which FIAT is based for
constructing those elements.  Our focus is on the case of
scalar-valued elements in affine spaces, although we indicate how
the techniques generalize on both counts.  We begin the rest
of the paper by recalling
definitions in \S~\ref{sec:prelim}.  The bulk of the
paper occurs in \S~\ref{sec:theory}, where we show how to map finite
element bases under affine equivalence, affine-interpolation
equivalence, and when neither holds.  We also sketch briefly how the
theory is adapted to the case of more general pullbacks such as
non-affine coordinate mappings or Piola transforms.  All the theory in
\S~\ref{sec:theory} assumes that the natural
pull-back operation (\ie~composition with coordinate change)
exactly preserves the function spaces between reference and physical
space.   However, in certain notable cases such as the Bell element,
this condition fails to hold.  In \S~\ref{sec:whatthebell}, we
give a more general theory with application to the Bell element.
Finally, in \S~\ref{sec:num}, we present some
numerical results using these elements.

\section{Definitions and preliminaries}
\label{sec:prelim}

Througout, we let $C_b^k(\Omega)$ denote the space of functions with
continuous and bounded derivatives up to and including order $k$ over
$\Omega$, and $C_b^k(\Omega)^\prime$ its topological dual.

\begin{definition}
  A \emph{finite element} is a triple $(K, P, N)$ such that
  \begin{itemize}
  \item $K \subset \mathbb{R}^d$ is a bounded domain.
  \item $P \subset C^k_b(K)$ for some integer $k \geq 0$ is a
    finite-dimensional function space.
  \item \(N = \{n_i\}_{i=1}^{\nu} \subset
    C^k_b(K)^\prime\) is a collection of
    linearly independent functionals whose actions restricted to $P$
    form a basis for $P^\prime$.
  \end{itemize}
\end{definition}

The nodes in $N$ are taken as objects in the full
infinite-dimensional dual, although sometimes we will only require
their restrictions to members of $P$.
For any \( n \in C^k_b(K)^\prime \),
define $\pi n \in P^\prime$ by restriction.  That is, define \( \pi n
(p) = n(p) \) for any  \( p \in P \). 

Further, with a slight abuse in notation, we
will let $N = \begin{bmatrix} n_1 & n_2 & \dots & n_\nu \end{bmatrix}^T$
denote a functional on $P^\nu$, or equivalently, a vector of $\nu$
members of the dual space.  

As shorthand, we define these spaces consisting of vectors of
functions or functionals by
\begin{equation}
  \label{eq:XXdag}
  \begin{split}
    X & \equiv \left( P \right)^\nu, \\
    X^\dagger & \equiv \left( C_b^k( K )^\prime \right)^\nu.
  \end{split}
\end{equation}

We can ``vectorize'' the restriction operator \( \pi \), so that 
for any 
\( N \in X^\dagger \), \( \pi N \in (P^\nu)^\prime \) has \( (\pi N)_i
= \pi(n_i) \).

Galerkin methods work in terms of a basis for the approximating space,
and these are typically built out of local bases for each element:
\begin{definition}
  Let \((K, P, N)\) be a finite element with $\dim P = \nu$.  The \emph{nodal basis} for $P$
  is the set \(\{\psi_i\}_{i=1}^{\nu}\) such that \(n_i(\psi_j) =
  \delta_{i,j}\) for each \(1 \leq i,j \leq \nu\).
\end{definition}

The nodal basis also can be written as \( X \ni \Psi = \begin{bmatrix} \psi_1 & \psi_2 & \dots &  \psi_\nu \end{bmatrix} \).

Traditionally, finite element codes construct the nodal basis for a
\emph{reference} finite element
\(\left(\hat{K}, \hat{P},
\hat{N}\right)\) and then map it into the basis for \( \left(K, P,
N\right) \) for each $K$ in the mesh.  Let $F:K \rightarrow \hat{K}$ be
the geometric mapping, as in Figure~\ref{fig:affmap}.  We let \( J \)
denote the Jacobian matrix of this transformation.

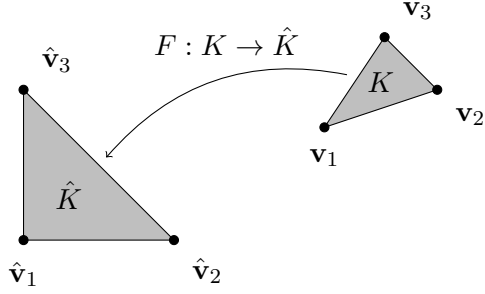
\begin{figure}
  \begin{center}
  \begin{tikzpicture}
    \draw[fill=lightgray] (0,0) coordinate (vhat1)
    -- (2,0) coordinate(vhat2)
    -- (0,2) coordinate(vhat3)--cycle;
    \foreach \pt\labpos\lab in {vhat1/below/\hat{\mathbf{v}}_1, vhat2/below right/\hat{\mathbf{v}}_2, vhat3/above right/\hat{\mathbf{v}}_3}{
      \filldraw (\pt) circle(.6mm) node[\labpos=1.5mm, fill=white]{$\lab$};
    }
    \draw[fill=lightgray] (4.0, 1.5) coordinate (v1)
    -- (5.5, 2.0) coordinate (v2)
    -- (4.8, 2.7) coordinate (v3) -- cycle;
    \foreach \pt\labpos\lab in {v1/below/\mathbf{v}_1, v2/below right/\mathbf{v}_2, v3/above right/\mathbf{v}_3}{
      \filldraw (\pt) circle(.6mm) node[\labpos=1.5mm, fill=white]{$\lab$};
    }
    \draw[<-] (1.1, 1.1) to[bend left] (4.3, 2.2);
    \node at (2.7, 2.65) {$F:K\rightarrow\hat{K}$};
    \node at (0.6,0.6) {$\hat{K}$};
    \node at (4.75, 2.1) {$K$};
  \end{tikzpicture}
  \end{center}
  \caption{Affine mapping to a reference cell \(\hat{K}\) from a
    typical cell \( K \).  Note that here $F$ maps from the physical
    cell $K$ to the reference cell $\hat{K}$ rather than the other way
  around.}
    \label{fig:affmap}
\end{figure}

Similarly to~\eqref{eq:XXdag}, we define the vector spaces relative to
the reference cell:
\begin{equation}
  \label{eq:XXdaghat}
  \begin{split}
    \hat{X} & \equiv \left( \hat{P} \right)^\nu, \\
    \hat{X}^\dagger & \equiv \left( C_b^k( \hat{K} )^\prime \right)^\nu.
  \end{split}
\end{equation}
As with \( \pi \), we define \( \hat{\pi} \hat{n} \) as the
restriction of \( \hat{n} \) to \( \hat{P} \), and can vectorize it over
\( \hat{X}^\dagger \) accordingly.

This geometric mapping induces a mapping between spaces of functions
over $K$ and $\hat{K}$ as well as between the dual spaces.  These are
called the pull-back, and push-forward operations, respectively:

\begin{definition}
  The \emph{pull-back} operation mapping
  \(
  C^k_b(\hat{K}) \rightarrow C^k_b(K)
  \)
  is defined by
  \begin{equation}
    \label{eq:pullback}
    F^*\left(\hat{f}\right) = \hat{f} \circ F
  \end{equation}
  for each \(\hat{f} \in C^k_b(\hat{K})\).
\end{definition}

\begin{definition}
  The \emph{push-forward} operation mapping the dual space
  \( C^k_b(K)^\prime\) into
  \( C^k_b(\hat{K})^\prime \) is defined by
  \begin{equation}
    F_*(n) = n\circ F^*
  \end{equation}
  for each \(n \in C^k_b(K)^\prime \).
\end{definition}

It is easy to verify that the pull-back and push-forward are linear operations
preserving the vector space operations.  Moreover, they are invertible
iff $F$ itself is.  Therefore, we have
\begin{proposition}
  \label{prop:iso}
  Given finite elements $(K,P,N)$ and
  $(\hat{K},\hat{P},\hat{N})$ such that $F(K)=\hat{K}$
  and $F^*(\hat{P}) = P$,
  $F^*:\hat{P} \rightarrow P$ and
  \( F_*:P^\prime \rightarrow \hat{P}^\prime \) are
  isomorphisms.
\end{proposition}

The pull-back and push-forward operations are also defined over the vector
spaces \( X \), \(X^\dagger\), \( \hat{X} \), and \( \hat{X}^\dagger
\). If $N$ is a vector of functionals and \( \Phi \) a vector of
functions, then the vector push-forward and pull-back are,
respectively 
\begin{equation}
  \begin{split}
    F_*(N) \in \hat{X}^\dagger, \ \ \ \left( F_*(N) \right)_i & = F_*(n_i), \\
    F^*(\hat{\Phi}) \in X, \ \ \ \left( F^*(\hat{\Phi}) \right)_i & = F^*(\hat{\phi_i}).
  \end{split}
\end{equation}

It will also be useful to consider vectors of functionals acting on
vectors of functions.  We define this to produce a matrix as follows.
If \(N = \begin{bmatrix} n_1 & n_2 & \dots & n_k \end{bmatrix}^T \) is
a collection of functionals and
\( \Phi = \begin{bmatrix} \phi_1 & \phi_2 & \dots &
  \phi_\ell \end{bmatrix}^T \) a collection of functions, then we
define the (outer) product $N(\Phi)$ to be the $k \times \ell$ matrix
\begin{equation}
  \label{eq:nonphi}
\left( N(\Phi) \right)_{ij} = n_i(\phi_j).
\end{equation}
For example, if $N$ is the vector of nodes of a finite element and
$\Psi$ contains the nodal basis functions, then the Kronecker delta
property is expressed as
\(
N(\Psi) = I.
\)

If $M$ is a matrix of numbers of appropriate shape and $\Phi \in X$ 
members of a function space $P$, then $M\Phi$ is just defined by
\(
(M \Phi)_i = \sum_{j=1}^\nu M_{ij} \Phi_j,
\)
according to the usual rule for matrix-vector multiplication.

\begin{lemma}
  Let $N \in X^\dagger$ and $\Phi \in X$ and \(M \in \mathbb{R}^{\nu
    \times \nu}\).  Then
  \begin{equation}
    N(M\Phi) = N(\Phi) M^T.
  \end{equation}
\end{lemma}
\begin{proof}
  The proof is a simple calculation:
  \[
  \left( N(M\Phi) \right)_{ij} = n_i \left( \left( M \Phi \right)_j
  \right)
  = n_i \left( \sum_{k=1}^\nu M_{jk} \phi_k \right)
  = \sum_{k=1}^\nu n_i \left( \phi_k\right)
  = \sum_{k=1}^\nu \left(N(\Phi)\right)_{ik} M_{jk}.
  \]
\end{proof}

The relationship between pull-back and push-forward also leads to
the vectorized relation
\begin{lemma}
  Let $N \in X^\dagger$ and $\hat{\Phi} \in \hat{X}$.  Then
  \begin{equation}
    N(F^*(\hat{\Phi})) = F_*(N)(\hat{\Phi})
  \end{equation}
\end{lemma}

\begin{definition}
  Let $(K,P,N)$ and $(\hat{K}, \hat{P}, \hat{N})$ be
  finite elements and $F$ an affine mapping on $K$.
  Then $(K,P,N)$ and $(\hat{K},\hat{P}, \hat{N})$ are \emph{affine equivalent} if
  \begin{itemize}
  \item $F(K) = \hat{K}$,
  \item The pullback maps $F^*(\hat{P}) = P$ (in the sense of equality
    of vector spaces),
  \item $F_*(N) = \hat{N}$ (in the sense of equality of finite sets).
  \end{itemize}
\end{definition}

\begin{definition}
  Let \( (K, P, N) \) be a finite element of class $C^k$ and
  \(\Psi \in X \) its nodal basis.  The \emph{nodal
    interpolant} \( \mathcal I_N: C_b^k(K) \rightarrow P \) is defined
  by 
  \begin{equation}
    \mathcal{I}(f) = \sum_{i=1}^\nu n_i(f) \psi_i.
  \end{equation}
\end{definition}
This interpolant plays a fundamental role in establishing
approximation properties of finite elements via the Bramble-Hilbert
Lemma~\citep{bramble1971bounds,dupontscott1980}. The
homogeneity arguments in fact go through for the following generalized
notion of element equivalence:
\begin{definition}
Two finite elements $(K, P, N)$ and $(K, P, \tilde{N})$ are
\emph{interpolation equivalent} if $\mathcal{I}_{N} = \mathcal{I}_{\tilde{N}}$.
\end{definition}

\begin{definition}
  If $(K, P, \tilde{N})$ is affine equivalent to $(\hat{K}, \hat{P}, \hat{N})$
  and interpolation equivalent to $(K, P, N)$, then
  $(K, P, N)$ and $(\hat{K}, \hat{P}, \hat{N})$ are
  \emph{affine-interpolation equivalent}.
\end{definition}

Brenner and Scott~\citep{BreSco} give the following result, of which
we shall make use:
\begin{proposition}
Finite elements \( (K, P, N) \) and \( (K, P, \tilde{N}) \) are
interpolation equivalent iff the spans of \( N \) and \( \tilde{N} \),
(viewed as subsets of \( C_b^k(K)^\prime \)), are equal.
\end{proposition}

For Lagrange and certain other finite elements, one simply has that
$F^*(\hat{\Psi}) = \Psi$, which allows for the traditional use of
reference elements used in FEniCS, Firedrake, and countless other codes.
However, for many other elements this is not the case.  It is our goal
in this paper to give a general approach that expresses $\Psi$ as a
linear transformation $M$ applied to $F^*(\hat{\Psi})$.

Before proceeding, we note that approximation theory for Argyris and
other families without affine-interpolation equivalence can proceed by
means of establishing the \emph{almost-affine}
property~\citep{ciarlet1978finite}.  Such proofs can involve
embedding the inequivalent element family into an equivalent one with
the requisite approximation properties.
For example, the Argyris element is proved almost-affine by comparison
to the ``type (5)'' quintic Hermite element.
Although we see definite computational consequences of
affine-equivalence, affine-interpolation equivalence, and neither
among our element families, we our approach to transforming
inequivalent families does not make use of any almost-affine properties.

\section{Transformation theory when \( F^*(\hat{P}) = P \)}
\label{sec:theory}
For now, we assume that the pull-back operation~\eqref{eq:pullback}
appropriately converts
the reference element function space into the physical function space
and discuss the construction of nodal bases based on relationships
between the reference nodes $\hat{N}$ and the pushed-forward
physical nodes $F_*(N)$.

We focus on the simplicial case, although generalizations do not have
a major effect, as we note later.  Throughout, we will
use following convention, developed in~\citep{RognesKirbyEtAl2009a}
for handling facet 
orientation in mixed methods but also useful in order higher-order
Lagrange degrees of freedom.  Since our examples are triangles
(2-simplices), it is not necessary to expand on the entire
convention.  Given a triangle with vertices
\( \left(\mathbf{v}_1, \mathbf{v}_2, \mathbf{v}_3\right) \), we
define edge $\gamma_i$ of the triangle to connect the vertices other
than $\mathbf{v}_i$.  The (unit) tangent vector \(
\mathbf{t}_i =  \begin{bmatrix}t^\mathbf{x}_i &
  t^\mathbf{y}_i\end{bmatrix}^T \), points in the direction from the lower- to the
higher-numbered vertex.
When triangles share an edge, then, they agree on its orientation.  The
normal to an edge is defined by rotating the tangent by applying
the matrix \( R = \begin{bmatrix}0 & 1 \\ -1 & 0 \end{bmatrix} \) so
that
\(\mathbf{n}_i = R \mathbf{t}_i = \begin{bmatrix} n^\mathbf{x}_i & n^\mathbf{y}_i\end{bmatrix}^T\)
We also let \( \mathbf{e}_i \) denote the midpoint of \( \gamma_i \).

Now, we fix some notation for describing nodes.  First, we
define $\delta_{\mathbf{x}}$ acting on any continuous function by pointwise
evaluation.  That is:
\begin{equation}
  \delta_{\mathbf{x}}(p) = p(\mathbf{x}).
\end{equation}
We let $\delta^{\mathbf{s}}_\mathbf{x}$ denote the
directional derivative in direction $\mathbf{s}$ at a point
$\mathbf{x}$, so that
\begin{equation}
  \delta^{\mathbf{s}}_\mathbf{x}(p)
  = \mathbf{s}^T \nabla p(\mathbf{x}).
\end{equation}
We use repeated superscripts to indicate higher-order derivatives,
so that \( \delta^{\mathbf{x}\mathbf{x}}_{\mathbf{x}} \) defines the second
directional derivative along the $x$-axis at point $\mathbf{x}$.

It will also be convenient to use block notation, with a single symbol
representing two or items.  For example, the gradient notation
\[
\nabla_{\mathbf{x}} = \begin{bmatrix}
  \delta^{\mathbf{x}}_{\mathbf{x}} &
  \delta^{\mathbf{y}}_{\mathbf{x}} \end{bmatrix}^T
\]
gives the pair of functionals evaluating the Cartesian derivatives
at a point $\mathbf{x}$.  To denote a gradient in a different basis, we
append the directions as superscripts so that
\[
\nabla^{\mathbf{nt}}_{\mathbf{x}} =
\begin{bmatrix} \delta^{\mathbf{n}}_{\mathbf{x}} &
  \delta^{\mathbf{t}}_{\mathbf{x}} \end{bmatrix}^T
\] contains the normal and tangential derivatives at a point
\(\mathbf{x}\). 

Similarly, we let
\[
\bigtriangleup_{\mathbf{v}} = \begin{bmatrix}
  \delta^{\mathbf{xx}}_{\mathbf{x}} & \delta^{\mathbf{xy}}_{\mathbf{x}}
  & \delta^{\mathbf{yy}}_{\mathbf{x}} \end{bmatrix}^T
\]
denote the vector of three functionals evaluating the unique 
(supposing sufficient smoothness) second partials at
\(\mathbf{x}\).

Let
\( \Psi = \{\psi_i\}_{i=1}^{\nu}\) be the nodal basis for
a finite element \((K, P, N)\) and
\( \hat{\Psi} = \{ \hat{\psi}_i \}_{i=1}^{\nu}\)
that for a reference element \( \left(\hat{K},\hat{P},\hat{N} \right)
\).  We also assume that \(F(K) = \hat{K}\) and \( F^*(\hat{P})=P \).
Because the pull-back is
invertible, it maps linearly independent sets to linearly independent
sets.  So, \(F^*(\hat{\Psi})\) must also be a basis for \(P\).
There exists an invertible
$\nu \times \nu$ matrix $M$ such that 
\begin{equation}
  \label{eq:M}
  \Psi = M F^*(\hat{\Psi}),
\end{equation}
or equivalently, that each nodal basis function is some linear
combination of the pull-backs of the reference nodal basis functions.

Our theory for transforming the basis functions (\ie~computing the
matrix $M$) will work via duality -- relating the matrix $M$ to how
the nodes, or at least their restrictions to the finite-dimensional
spaces, push forward.

It will be useful to define as an intermediate $\nu \times \nu$ matrix
\( B=F_*(N)(\hat{\Psi}) \).  Recall from~\eqref{eq:nonphi}
that its
entries for $1 \leq i, j \leq \nu$ are
\begin{equation}
  \label{eq:Fmat}
  B_{ij} \equiv F_*(n_i)(\hat{\psi}_j) = n_i(F^*(\hat{\psi}_j))
\end{equation}
This matrix, having nodes only applied to members of \( P \)  is
indifferent to restrictions and so
\( B = F_*(\pi N)(\hat{\Psi}) \) as well.

Because of Proposition~\ref{prop:iso} and finite-dimensionality, the 
the nodal sets
\( \hat{\pi}\hat{N} \) and \( F_*(\pi N) \)
are both bases for $\hat{P}^\prime$, and so there
exists an invertible $\nu \times \nu$ matrix $V$ such that
\begin{equation}
  \label{eq:V}
  \hat{\pi} \hat{N} = V F_*(\pi N)
\end{equation}

Frequently, it may be easier to express the pushed-forward nodes as
a linear combination of the reference nodes.  In this case, one
obtains the matrix $V^{-1}$.  At any rate, the matrices \( V \) and \(
M \) are closely related.
\begin{theorem}
  For finite elements $(K,P,N)$ and
  $(\hat{K},\hat{P},\hat{N})$ with $F(K)=\hat{K}$ and
  $F_*(\hat{P}) = P$, the matrices in~\eqref{eq:M}
  and~\eqref{eq:V} satisfy
  \begin{equation}
    M = V^T.
  \end{equation}
  \label{thm:MVt}
\end{theorem}
\begin{proof}
  We proceed by relating both matrices to $B$ defined
  in~\eqref{eq:Fmat} via the Kronecker property of nodal bases.
  First, we have
  \[
    I = N(\Psi) = N(MF^*(\hat{\Psi})) = N(F^*(\hat{\Psi})) M^T
    = B M^T.
  \]
  so that
  \begin{equation}
    M = B^{-T}.
  \end{equation}
  Similarly,
  \[
  I = \left(VF_*(N)\right)(\hat{\Psi})
    = V F_*(N)(\hat{\Psi})
    = V B,
  \]
  so that \(V=B^{-1}\)
  and the result follows.
\end{proof}

That is, to relate the pullback of the reference element basis
functions to any element's basis functions, it is sufficient to
determine the relationship between the nodes.

\subsection{Affine equivalence: The Lagrange element}
\label{sec:lagrange}
When elements form affine-equivalent families, the matrix $M$ has a
particularly simple form.
\begin{theorem}
  \label{thm:aeI}
If $(K,P,N)$ and
  $(\hat{K},\hat{P},\hat{N})$ are affine-equivalent finite
elements then the transformation matrix $M$ is the identity.
\end{theorem}
\begin{proof}
  Suppose the two elements are affine-equivalent, so that $F_*(N) = \hat{N}$.  Then, a direct calculation gives
  \[
    N(F^*(\hat{\Psi})) = F_*(N)(\hat{\Psi}) = \hat{N}(\hat{\Psi}) = I
  \]
  so that $M=I$.
\end{proof}

The \emph{Lagrange} elements are the most widely used finite elements
and form the prototypical affine-equivalent family~\citep{BreSco}.   For a simplex
$K$ in dimension $d$ and integer $r \geq 1$, one defines $P = P_r(K)$ to be the space
of polynomials over $K$ of total degree no greater than $r$, which has
dimension $\binom{r+d}{d}$.  The nodes are taken to be pointwise
evaluation at a lattice of $\binom{r+d}{d}$ points.  Classically,
these are taken to be regular and equispaced, although options with
superior interpolation and conditioning properties for large $r$ are
also known~\cite{HesWar}.  One must ensure that nodal locations are
chosen at the boundary to enable $C^0$ continuity between adjacent
elements.  A cubic Lagrange triangle ($r=3$ and $d=2$) is shown
earlier in Figure~\ref{lag3}.

The practical effect of Theorem~\ref{thm:aeI} is that the reference
element paradigm ``works.''  That is, a computer code contains a
routine to evaluate the nodal basis \(
\hat{\Psi} \) and its derivatives for a reference
element \( (\hat{K}, \hat{P}, \hat{N}) \).  Then, this routine is
called at a set of quadrature points in \( \hat{K} \).  One obtains
values of the nodal basis at quadrature points on each cell \( K \) by
pull-back, so no additional work is required.  To obtain the gradients of
each basis function at each quadrature point, one simply multiplies
each basis gradient at each point by \( J^T \).

On the other hand, when \( M \neq I \), the usage of tabulated
reference values is more complex.  Given a table
\begin{equation}
  \label{eq:varpsi}
  \hat{\varPsi}_{iq} = \hat{\psi}_i(\hat{\xi}_q)
\end{equation}
of the reference basis at the reference quadrature points,
one finds the nodal basis for \( (K, P, N) \) by constructing \( M \)
for that element and then computing the matrix-vector product
\( M \hat{\varPsi} \) so that
\begin{equation}
  \label{eq:psiiq}
  \psi_i(\xi_q) = \sum_{k=1}^{\nu} M_{i,k} \hat{\varPsi}_{k,q}
\end{equation}

Mapping gradients from the reference element requires both
multiplication by \( M \) as well as application of \( J^T \) by the
chain rule.  We define \( D\hat{\varPsi} \in \mathbb{R}^{\nu \times |\xi|
  \times 2} \) by
\begin{equation}
  D\hat{\varPsi}_{i,q,:} = \hat{\nabla}\hat{\psi}_i(\hat{\xi})_q.
\end{equation}

Then, the basis gradients requires contraction with \( M \)
\begin{equation}
  \label{eq:dvarpsiprime}
  D\varPsi^\prime_{i,q,:} := \sum_{k=1}^{\nu} M_{i,k} D\hat{\varPsi}_{k,q,:},
\end{equation}
followed by the chain rule
\begin{equation}
  \label{eq:dvarpsi}
  D\varPsi_{i,q,:} := J^T D\varPsi^\prime_{i,q,:}.
\end{equation}
In fact, the application of \( M \) and \( J^T \) can be
performed in either order.  Note that applying \( M \) requires an \(
\nu \times \nu \) 
matrix-vector multiplication and in principle couples all basis functions
together, while applying \( J^T \) works pointwise on each basis
function separately.  When \( M \) is quite sparse, one expects this
to be a small additional cost compared to the other required
arithmetic.  We present further details for this in the case of Hermite
elements, to which we now turn.

\subsection{The Hermite element: affine-interpolation equivalence}
The Hermite triangle~\cite{ciarlet1972general}, show in
Figure~\ref{herm3} is based
cubic polynomials, although higher-order instances can also be
defined~\citep{BreSco}.  In contrast to the 
Lagrange element, its node set includes function values and
derivatives at the nodes, as well as an interior function value.
The resulting finite element spaces have $C^0$ continuity with 
$C^1$ continuity at vertices.  
They provide a classic example of elements that are not
affine equivalent but instead give affine-interpolation equivalent
families. 

We will let $(K,P,N)$ be a cubic Hermite triangle, specifying the
gradient at each vertex in terms of the Cartesian derivatives -- see
Figure~\ref{physherm}.  Let $\{\mathbf{v}_i\}_{i=1}^3$ be the three vertices
of $K$ and $\mathbf{v}_4$ its barycenter.  We order the nodes $N$ by
\begin{equation}
  N =
  \begin{bmatrix}
    \delta_{\mathbf{v}_1} &
    \nabla_{\mathbf{v}_1}^T &
    \delta_{\mathbf{v}_2} &
    \nabla_{\mathbf{v}_2}^T &
    \delta_{\mathbf{v}_3} &
    \nabla_{\mathbf{v}_3}^T &
    \delta_{\mathbf{v}_4}
  \end{bmatrix}^T,
\end{equation}
using block notation.

\begin{figure}
  \begin{subfigure}[t]{0.4\textwidth}
    \begin{center}
  \begin{tikzpicture}[scale=2.0] 
    \draw[fill=cyan] (0,0) -- (1, 0) -- (0, 1) -- cycle;
    \foreach \i/\j in {0/0, 1/0, 0/1}{
      \draw[fill=black] (\i, \j) circle (0.02);
      \draw[->] (\i, \j) -- (\i, \j + 0.1);
      \draw[->] (\i, \j) -- (\i+0.1, \j);
    }
    \draw[fill=black] (1/3, 1/3) circle (0.02);
  \end{tikzpicture}
  \end{center}
  \caption{Reference Hermite element}
  \label{refherm}
  \end{subfigure}
  \hfill
  \begin{subfigure}[t]{0.4\textwidth}
    \begin{center}
      \begin{tikzpicture}[scale=2.0]
    \draw[fill=cyan] (0,0) -- (1.5, 0.5) -- (0.8, 1.2) -- cycle;
    \foreach \i/\j in {0/0, 1.5/0.5, 0.8/1.2}{
      \draw[fill=black] (\i, \j) circle (0.02);
      \draw[->] (\i, \j) -- (\i, \j + 0.1);
      \draw[->] (\i, \j) -- (\i+0.1, \j);
    }
    \draw[fill=black] (0.7667, 0.5667) circle (0.02);        
    \end{tikzpicture}
    \end{center}
    \caption{Physical Hermite element}
    \label{physherm}
  \end{subfigure}
  \caption{Reference and physical cubic Hermite elements with
    gradient degrees of freedom expressed in terms of local Cartesian
    directional derivatives.}
  \label{fig:hermrefandphys}
\end{figure}
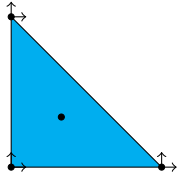
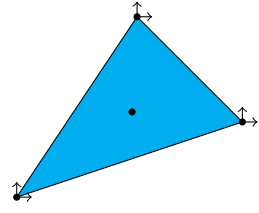

Now, we fix the reference element $(\hat{K}, \hat{P}, \hat{N})$ with
\( \hat{K} \) as the unit right triangle and express the gradient by the
derivatives in the direction of the reference Cartesian coordinates, as in
Figure~\ref{refherm}.
Let \( \{\hat{\mathbf{v}}_i\}_{i=1}^3 \) be the three vertices
of \( \hat{K} \) and \( \hat{\mathbf{v}}_4 \) its barycenter.
We define \( \hat{N} \) analogously to \( N \).

Consider the relationship between the nodal basis
functions $\Psi$ and the pulled-back $F^*(\hat{\Psi})$.  For any
$\hat{\psi} \in \hat{P}$, the chain rule leads to
\begin{equation}
  \label{eq:justthechainrule}
\nabla (\hat{\psi} \circ F) = J^T \hat{\nabla} \hat{\psi} \circ F.
\end{equation}

Now, suppose that $\hat{\psi}$ is a nodal basis function corresponding to
evaluation at a vertex or the barycenter, so that
\( \delta_{\hat{\mathbf{v}}_i}\hat{\psi} = 1 \) for some \( 1 \leq i \leq
4 \), with  the remaining reference nodes vanishing on \( \hat{\psi} \).
We compute that
\[
\delta_{\mathbf{v}_i} F^*(\hat{\psi})
= ( \hat{\psi} \circ F )\left( \mathbf{v}_i \right)
= \hat{\psi}(\hat{v}_i) = 1,
\]
while
\( \delta_{\mathbf{v}_j} F^*(\hat{\psi}) = 0 \)
for \( 1 \leq j \leq 4 \) with \( j \neq i \).
Also, since the reference gradient of $\hat{\psi}$ vanishes at each
vertex,~\eqref{eq:justthechainrule} implies that the physical gradient
of \( F^*(\hat{\psi}) \) must also vanish at each vertex.  So, pulling back
\( \hat{\psi} \) gives the corresponding nodal basis function for \(
(K, P, N) \).

The situation changes for the derivative basis functions.  Now
take $\hat{\psi}$ to be the basis function with unit-valued derivative
in, say, the \( \hat{\mathbf{x}} \) direction at vertex \(
\hat{\mathbf{v}}_i \) and other degrees of freedom vanishing.  Since
it vanishes at each vertex and the barycenter of \( \hat{K} \),
\( F^*(\hat{\psi}) \) will vanish at each vertex  and the barycenter of \(
K \).  The reference gradient of \(\hat{\psi}\) vanishes at the vertices other
than \( i \), so the physical gradient of its pullback must also
vanish at the corresponding vertices of \( K \).
However,~\eqref{eq:justthechainrule} shows that
\( \nabla(\hat{\psi} \circ F) \) will typically not yield
\( \begin{bmatrix} 1 & 0 \end{bmatrix}^T \) at \( \mathbf{v}_i \).
Consequently, the pull-backs of the reference derivative basis
functions do not produce the physical
basis functions.

Equivalently, we may express this failure in terms of the nodes --
pushing forward \( N \) does not yield \( \hat{N} \).  We
demonstrate this pictorially in Figure~\ref{fig:hermpushforward},
showing the images of the derivative nodes under push-forward do not
correspond to the reference derivative nodes.  Taking this view allows us
to address the issue using Theorem~\ref{thm:MVt}.

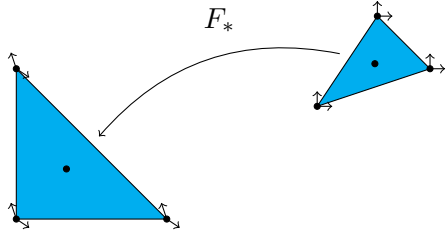
\begin{figure}
  \begin{center}
  \begin{tikzpicture}
    \draw[fill=cyan] (0,0) -- (2,0) -- (0,2) -- cycle;
    \foreach \x\y in {0/0, 2/0, 0/2}{
      \filldraw (\x,\y) circle(0.04);
      \draw[->] (\x,\y) -- (\x+6/7*0.2,\y-4/7*0.2);
      \draw[->] (\x,\y) -- (\x-5/14*0.2, \y+15/14*0.2);
    }
    \draw[fill=black] (2/3, 2/3) circle(0.04);

    \draw[fill=cyan] (4.0, 1.5) -- (5.5, 2.0) -- (4.8, 2.7) -- cycle;
    \foreach \x\y in {4/1.5, 5.5/2, 4.8/2.7}{
      \filldraw (\x,\y) circle(.04);
      \draw[->] (\x,\y)--(\x,\y+0.2);
      \draw[->] (\x,\y)--(\x+0.2,\y);
    }
    \filldraw (4.767, 2.067) circle (0.04);
    \draw[<-] (1.1, 1.1) to[bend left] (4.3, 2.2);
    \node at (2.7, 2.65) {$F_*$};
    \end{tikzpicture}
  \end{center}
  \caption{Pushing forward the Hermite derivative nodes in physical
    space does \emph{not} produce the reference derivative nodes.}
\label{fig:hermpushforward}
\end{figure}

This discussion using the chain rule can be summarized by the
matrix-valued equation
\begin{equation}
  F_*(N) =
  \begin{bmatrix}
    1 & 0 & 0 & 0 & 0 & 0 & 0 \\
    0 & J^T & 0 & 0 & 0 & 0 & 0 \\
    0 & 0 & 1 & 0 & 0 & 0 & 0 \\
    0 & 0 & 0 & J^T & 0 & 0 & 0 \\
    0 & 0 & 0 & 0 & 1 & 0 & 0 \\
    0 & 0 & 0 & 0 & 0 & J^T & 0 \\
    0 & 0 & 0 & 0 & 0 & 0 & 1
    \end{bmatrix}
  \hat{N},
  \label{eq:blockguy}
\end{equation}
noting that the second, fourth, and sixth rows and columns of this
matrix are blocks of two, and each ``$0$'' is taken to be the zero
matrix of appropriate size.  This is exactly the inverse of \( V \)
from Theorem~\ref{thm:MVt}.

In this case, the transformation $V$ is quite local --
that is, only the push-forward of nodes at a given point are used to
construct the reference nodes at the image of that point.  This seems
to be generally true for interpolation-equivalent elements, although
functionals with broader support (e.g. integral moments over the
cell or a facet thereof) would require a slight adaptation.
We will see presently for Morley and Argyris elements
that the transformation neeed not be block diagonal for elements
without interpolation equivalence. 
At any rate, the following elementary observation from linear algebra
suggests the sparsity of \( V \):
\begin{proposition}
  \label{prop:span}
  Let \( W \) be a vector space with sets of vectors
  \( W_1 = \{ w^1_i \}_{i=1}^m
  \subset W \) and \( W_2 = \{ w^2_i \}_{i=1}^n \).  Suppose that
  \( \mathrm{span} W_1 \subset \mathrm{span W_2} \) so that there exists a
  matrix \( A \in \mathbb{R}^{m \times n} \) such that \(w^1_i =
  \sum_{k=1}^n A_{ik} w^2_k \).  If we further have that some \( w^1_i
  \in \mathrm{span} \{ w^2_j \}_{j\in \mathcal{J}} \) for some \(
  \mathcal{J} \subset [1, n] \), then \( A_{ij} = 0 \) for all \( j
  \notin \mathcal{J} \).
\end{proposition}

Our theory applies equally to the
general family of Hermite triangles of degree \( k \geq 3\).
In those cases, the nodes consist of gradients at vertices
together with point-wise values at appropriate places.  All
higher-order cases generate $C^0$ families of elements with
$C^1$-continuity at vertices.  The $V$ matrix remains 
analogous to the cubic case, with $J^{-T}$ on the diagonal in three places
corresponding to the vertex derivative nodes.  No major differences
appear for the tetradral Hermite elements, either.

As we saw earlier, Hermite and other elements for which \( M \neq I \)
incur an additional cost in mapping from the reference element, as one
must compute basis function values and gradients via~\eqref{eq:psiiq}
and~\eqref{eq:dvarpsi}.  The key driver of this additional cost is the
application of \( M \).  Since \( M \) is very sparse for Hermite
elements -- just 12 nonzeros counting the 1's on the diagonal --
evaluating~\eqref{eq:psiiq} requires just \( 12 \) operations per
column, so a 10-point quadrature rule requires 120 operations.
Evaluating~\eqref{eq:dvarpsiprime} requires twice this,
or 240 operations.  Applying \( J^T \) in~\eqref{eq:dvarpsi} is
required whether Hermite or Lagrange elements are used.  It requires
$4 \times 10$ times the number of quadrature points used -- so a
10-point rule
would require 400 operations.  Hence, the chain rule
costs more than the application of \( M \) in this situation.
On the other hand, building an element stiffness matrix requires a
double loop over these 10 basis functions nested with a loop over the,
say, 10 quadrature points.  Hence, the loop body requires 1000
iterations, and with even a handful of operations will easily dominate 
the additional cost of multiplying by \( M \).

\subsection{The Morley and Argyris elements}
The construction of \(C^1\) finite elements, required for problems
such as plate bending or the Cahn-Hilliard equations, is a
long-standing difficulty.  Although it
is possible to work around this requirement by rewriting the
fourth-order problem as a lower order system or by using \(
C^0 \) elements in conjunction with variational form penalizing the
jumps in derivatives~\citep{engel2002continuous, wells2007c}, this
doesn't actually give a \( C^1 \) solution.

The quadratic Morley triangle~\citep{morley1971constant}, shown in
Figure~\ref{morley}, finds application in plate-bending problems and
also provides a relatively simple motivation for and application of
the theory developed here.
The six degrees of freedom, vertex values and
the normal derivatives on each edge midpoint, lead to an assembled finite
element space that is neither \( C^0 \) nor \( C^1 \), but it is still
suitable as a convergent nonconforming approximation for fourth-order
problems.

The quintic Argyris triangle~\citep{argyris1968tuba}, shown in
Figure~\ref{arg}, with its 21 degrees, gives a proper \( C^1
\) finite element.
Hence it can be used generically for fourth-order problems as well as
second-order problems for which a continuously differentiable solution
is desired.  The Argyris elements use the
values, gradients, and second derivatives at each triangle vertex plus
the normal derivatives at edge midpoints as the twenty-one
degrees of freedom.

It has been suggested that the Bell element~\citep{bell1969refined}
represents a simpler \( C^1 \) element than the Argyris element, on
the account that it has fewer degrees of freedom.  Shown in
Figure~\ref{bell}, we see that the edge normal derivatives have been
removed from the Argyris element.  However, this comes with a
(smaller but) more complicated function space.  Rather than full
quintic polynomials, the Bell element uses quintic polynomials
that have normal derivatives on each edge of only third degree.  This
constraint on the polynomial space turns out to complicate the
transformation of Bell elements compared to Hermite or even Argyris.
For the rest of this section, we focus on Morley and Argyris,
returning to Bell later.

It can readily be seen that, like the Hermite element, the standard
affine mapping will not preserve nodal bases.  Unlike
the Hermite element, however, the Morley and Argyris elements do not form
affine-interpolation equivalent families -- the spans of the nodes are
not preserved under push-forward thanks to the edge normal derivatives
-- see Figure~\ref{fig:morleypushforward}.
As the Morley and Aryris nodal sets do not contain a full gradient at
edge midpoints, the technique used for Hermite elements cannot be
directly applied.

\begin{figure}
  \begin{center}
  \begin{tikzpicture}
    \draw[fill=green] (0,0) -- (2,0) -- (0,2) -- cycle;
    \foreach \x\y in {0/0, 2/0, 0/2}{
      \filldraw (\x,\y) circle(0.04);
    }
    \foreach \x\y\nx\ny in {1/1/.3536/.3536,0/1/-.9113/1.070,1/0/.6010/-1.1971}{
      \draw[->] (\x, \y) -- (\x+0.2*\nx, \y+0.2*\ny);
      }
    \draw[fill=green] (4.0, 1.5) -- (5.5, 2.0) -- (4.8, 2.7) -- cycle;
    \foreach \x\y in {4/1.5, 5.5/2, 4.8/2.7}{
      \filldraw (\x,\y) circle(.04);
    }
    \foreach \x\y\nx\ny in {5.15/2.35/.707/.707,4.4/2.1/-0.8321/0.5547,4.75/1.75/.3163/-.9487}{
      \draw[->] (\x,\y) -- (\x+0.2*\nx,\y+0.2*\ny);
      }
    \draw[<-] (1.2, 1.2) to[bend left] (4.0, 2.2);
    \node at (2.7, 2.65) {$F_*$};
    \end{tikzpicture}
  \end{center}
  \caption{Pushing forward the Morley derivative nodes in physical
    space does \emph{not} produce the reference derivative nodes.}
\label{fig:morleypushforward}
\end{figure}
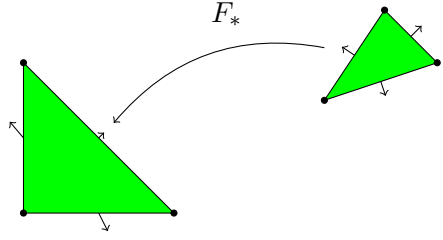

To work around this, we introduce the following idea:
\begin{definition}
  Let $(K,P,N)$ and $(\hat{K}, \hat{P}, \hat{N})$ be
  finite elements of class $C^k$ with affine mapping $F:K\rightarrow
  \hat{K}$ and associated pull-back and push-forward $F^*$ and $F_*$.
  Suppose also that \( F^*(\hat{P}) = P \).
  Let
  \(N^c = \left\{n^c_i\right\}_{i=1}^\mu \subset C_b^k(K)^\prime \)
  and
  \( \hat{N}^c = \left\{ \hat{n}^n_i \right\}_{i=1}^\mu \subset
  C^k(\hat{K})^\prime \) be
  such that
  \begin{itemize}
  \item \( N \subset N^c \) (taken as sets rather than vectors),
  \item \( \hat{N} \subset \hat{N}^c \) (again as sets),
  \item \( \mathrm{span}(F_*(N^c)) = \mathrm{span}(\hat{N}^c) \) in
    \( C^k(\hat{K})^\prime \).
  \end{itemize}
  Then \( N^c \) and \( \hat{N}^c \) form a 
  \emph{compatible nodal completion} of \( N \) and \( \hat{N} \).
\end{definition}

\begin{example}
  Let $(K,P,N)$ and $(\hat{K}, \hat{P}, \hat{N})$ be the
  Morley triangle and reference triangle.  Take $N^c$ to
  contain all the nodes of $N$ together with the tangential derivatives at
  the midpoint of each edge of $K$ and similarly for $\hat{N}^c$.
  In this case, $\mu = 9$.  Then, both $N^c$ and $\hat{N}^c$
  contain complete gradients at each edge midpoint and function values
  at each vertex.  The push-forward of $N^c$ has the same
  span as $\hat{N}^c$ and so $N^c$ and $\hat{N}^c$ form a
  compatible nodal completion of \( N \) and \( \hat{N} \).  This is shown
  pictorially in Figure~\ref{fig:morleybridge}.
\end{example}

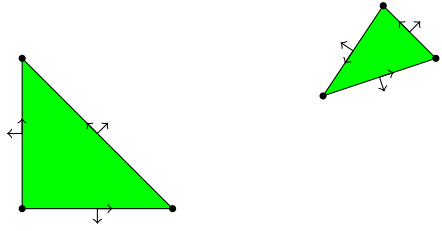
\begin{figure}
    \begin{center}
  \begin{tikzpicture}
    \draw[fill=green] (0,0) -- (2,0) -- (0,2) -- cycle;
    \foreach \x\y in {0/0, 2/0, 0/2}{
      \filldraw (\x,\y) circle(0.04);
    }
    \foreach \x\y\nx\ny\tx\ty in {1/1/.7071/.7071/-.7071/.7071,
      0/1/-1/0/0/1,
      1/0/0/-1/1/0}
    {
      \draw[->] (\x, \y) -- (\x+0.2*\nx, \y+0.2*\ny);
      \draw[->] (\x, \y) -- (\x+0.2*\tx, \y+0.2*\ty);
      }
    \draw[fill=green] (4.0, 1.5) -- (5.5, 2.0) -- (4.8, 2.7) -- cycle;
    \foreach \x\y in {4/1.5, 5.5/2, 4.8/2.7}{
      \filldraw (\x,\y) circle(.04);
    }
    \foreach \x\y\nx\ny\tx\ty in {5.15/2.35/.707/.707/-.707/.707,4.4/2.1/-0.8321/0.5547/-.5547/-.8321,4.75/1.75/.3163/-.9487/.9487/.3163}{
      \draw[->] (\x,\y) -- (\x+0.2*\nx,\y+0.2*\ny);
      \draw[->] (\x,\y) -- (\x+0.2*\tx,\y+0.2*\ty);
      }
    \end{tikzpicture}
  \end{center}
  \caption{Nodal sets \( \hat{N}^c \) and \( N^c \) giving the
    compatible nodal completion of $N$ and $\hat{N}$ for a Morley
    element and reference element are formed by including tangential
    derivatives along with normal derivatives at each edge midpoint.}
\label{fig:morleybridge}
\end{figure}

A similar completion -- supplementing the nodes with tangential
derivatives at edge midpoints -- exists for the Argyris nodes and
reference nodes~\citep{dominguez2008algorithm}.

Now, since the spans of $\hat{N}^c$ and $F_*(N^c)$ agree (even in
\( C_b^k(\hat{K})^\prime \)), there exists a $\mu \times \mu$ matrix
$V^c$, typically block diagonal, such that
\begin{equation}
  \label{eq:vc}
  \hat{N}^c = V^c F_*(N^c).
\end{equation}
Let \( E \in \mathbb{R}^{\nu \times \mu} \) be the
Boolean matrix with \( E_{ij} = 1 \) iff \( \hat{n}_i = \hat{n}_j^c
\) so that
\begin{equation}
  \label{eq:e}
  \hat{N} = E \hat{N}^c,
\end{equation}
and it is clear that
\begin{equation}
  \hat{N} = E V^c F_*(N^c).
\end{equation}
That is, the reference nodes are linear combinations of the
pushed-forward nodes \emph{and} the extended nodes,
but we must have the linear combination in terms of the pushed-forward
nodes alone.

Recall that building the nodal basis only requires the action of the
nodes on the polynomial space. 
Because \( \mu > \nu \), the set of nodes \( \pi N^c \) must be
linearly dependent. 
So, we seek a matrix \( D \in \mathbb{R}^{\mu \times \nu} \) such that
\begin{equation}
  \label{eq:d}
  \pi N^c = D \pi N.
\end{equation}
Since \( F_* \) is an isomorphism, such a \( D \) also gives
\begin{equation}
  \hat{\pi} F_*(N^c) = D \hat{\pi} F_*(N).
\end{equation}

Rows \( i \) of the matrix \( D \) such that \( n^c_i = n_j \) for 
some \( j \) will just have \( D_{ik} = \delta_{kj} \) for \( 1 \leq k
\leq \nu \).  The remaining rows must be constructed somehow via
an interpolation argument, although the details will vary by element.

This discussion suggests a three-stage process, each encoded by matrix
multiplication, for converting the push-forwards of the physical nodes
to the reference nodes, hence giving a factored form of $V$
in~\eqref{eq:V}.  Before working examples, we summarize this in the
following theorem:
\begin{theorem}
  Let $(K,P,N)$ and $(\hat{K}, \hat{P}, \hat{N})$ be
  finite elements with affine mapping $F:K\rightarrow \hat{K}$ and
  suppose that $F^*(\hat{P}) = P$.  Let $N^c$ and $\hat{N}^c$
  be a compatible nodal completion of \( N \) and \( \hat{N} \).
  Then given matrices $E \in \mathbb{R}^{\nu \times \mu}$ from~\eqref{eq:e},
  $V^c \in \mathbb{R}^{\mu \times \mu}$ from~\eqref{eq:vc}
  and $D \in \mathbb{R}^{\mu \times \nu}$ from~\eqref{eq:d} that
  builds the (restrictions of) the extended nodes out of the given
  physical nodes, the nodal transformation matrix $V$ satisfies
  \begin{equation}
    V = E V^C D.
  \end{equation}
\end{theorem}

This gives a general outline for mapping finite elements, and
we illustrate now by turning to the Morley element.

\subsubsection{The Morley element}
Following our earlier notation for the geometry and nodes, we order
the nodes of a Morley triangle by
\begin{equation}
  \label{eq:morleyN}
  N =
   \begin{bmatrix} \delta_{\mathbf{v}_1} & \delta_{\mathbf{v}_2}
    & \delta_{\mathbf{v}_3} &
    \delta^{\mathbf{n}_1}_{\mathbf{e}_1} &
    \delta^{\mathbf{n}_2}_{\mathbf{e}_2} &
    \delta^{\mathbf{n}_3}_{\mathbf{e}_3}
  \end{bmatrix}^T
\end{equation}

Nodes $N^C$ will also include tangential derivatives
at the edge midpoint.  We put
\begin{equation}
  \label{eq:morleyNc}
  N^c
  = \begin{bmatrix} \delta_{\mathbf{v}_1} & \delta_{\mathbf{v}_2}
    & \delta_{\mathbf{v}_3} &
    (\nabla^{\mathbf{n}_1\mathbf{t}_1}_{\mathbf{e}_1})^T &
    (\nabla^{\mathbf{n}_2\mathbf{t}_2}_{\mathbf{e}_2})^T &
    (\nabla^{\mathbf{n}_3\mathbf{t}_3}_{\mathbf{e}_3})^T
  \end{bmatrix}^T,
\end{equation}
Again, this is a block vector the last three entries each consist of
two values.  We give the same
ordering of reference element nodes \( \hat{N} \) and \( \hat{N}^c \).

The matrix $E$ simply extracts the members of $N^C$ that are also in
$N$, so with $\eta = \begin{bmatrix} 1 & 0 \end{bmatrix}$, we have the
block matrix
\begin{equation}
  E = \begin{bmatrix}
    1 & 0 & 0 & 0 & 0 & 0 \\
    0 & 1 & 0 & 0 & 0 & 0 \\
    0 & 0 & 1 & 0 & 0 & 0 \\
    0 & 0 & 0 & \eta & 0 & 0 \\
    0 & 0 & 0 & 0 & \eta & 0 \\
    0 & 0 & 0 & 0 & 0 & \eta
  \end{bmatrix}.
\end{equation}

Because the gradient nodes in \( N^c \) use normal and tangential
coordinates, \( V^c \) will be slightly more more complicated than \(
V \) for the Hermite element.
For local edge \( \gamma_i \), we define the (orthogonal) matrix
\[
G_i = \begin{bmatrix} \mathbf{n}_i & \mathbf{t}_i \end{bmatrix}^T
\]
with the normal and tangent vector in the rows.  Similarly, we let
\[
\hat{G}_i = \begin{bmatrix} \hat{\mathbf{n}}_i & \hat{\mathbf{t}}_i \end{bmatrix}^T
\]
contain the unit normal and tangent to edge \( \hat{\gamma}_i \) of the reference cell
\( \hat{K} \).  It is clear that
\begin{equation}
  F_*(\nabla^{\mathbf{n}_i \mathbf{t}_i}_{\mathbf{e}_i})
  = F_*(G_i \nabla_{\mathbf{e}_i}) = G_i F_*(\nabla_{\mathbf{e}_i})
  = G_i J^T \hat{\nabla}_{\mathbf{e}_i}
  = G_i J^T \hat{G}_i^T \hat{\nabla}^{\hat{\mathbf{n}}_i\hat{\mathbf{t}}_i}_{\hat{\mathbf{e}}_i},
\end{equation}
so, defining
\begin{equation}
  \label{eq:B}
  B^i = (G_i J^T \hat{G}_i^T)^{-1} = \hat{G}_i J^{-T} G_i^T,
\end{equation}
we have that
\begin{equation}
  V^C =
  \begin{bmatrix}
    1 & 0 & 0 & 0 & 0 & 0 \\
    0 & 1 & 0 & 0 & 0 & 0 \\
    0 & 0 & 1 & 0 & 0 & 0 \\
    0 & 0 & 0 & B^1 & 0 & 0 \\
    0 & 0 & 0 & 0 & B^2 & 0 \\
    0 & 0 & 0 & 0 & 0 & B^3 \\
  \end{bmatrix}.
\end{equation}

Now, we turn to the matrix \( D \in \mathbb{R}^{9 \times 6} \),
writing members of \( \pi N^c \) in terms of \( \pi N \) alone.  The challenge
is to express the tangential derivative nodes in terms of the
remaining six nodes -- vertex values and normal derivatives.  In fact,
only the vertex values are needed.  Along any edge, any member of $P$
is just a univariate quadratic polynomial, and so the tangential
derivative is linear.  Linear functions attain their
average value over an interval at its midpoint.  But the average value of
the derivative over the edge is just the difference between vertex
values divided by the edge length.  The matrix $D$ must be
\begin{equation}
  D = \begin{bmatrix}
    1 & 0 & 0 & 0 & 0 & 0 \\
    0 & 1 & 0 & 0 & 0 & 0 \\
    0 & 0 & 1 & 0 & 0 & 0 \\
    0 & 0 & 0 & 1 & 0 & 0 \\
    0 & -\ell_1^{-1} & \ell_1^{-1} & 0 & 0 & 0 \\
    0 & 0 & 0 & 0 & 1 & 0 \\
    -\ell_2^{-1} & 0 & \ell_2^{-1} & 0 & 0 & 0 \\
    0 & 0 & 0 & 0 & 0 & 1 \\
    -\ell_3^{-1} & \ell_3^{-1} & 0 & 0 & 0 & 0 \\
    \end{bmatrix}
\end{equation}
We can also arrive at this formulation of \( D \) in another way, that
sets up the discussion used for Argyris and later Bell elements.
Consider the following univariate result:
\begin{proposition}
  Let \( p(x) \) any quadratic polynomial on \(
  [-1, 1] \).  Then
  \begin{equation}
    p^\prime(0) = \tfrac{1}{2} \left( p(1) - p(-1) \right)
  \end{equation}
\end{proposition}
\begin{proof}
  Write \( p(x) = a + b x + c x^2 \).  Then
  \( p^\prime(x) = b + 2 c x \) so that \( p^\prime(0) = b \).
  Also note that
  \( p(1) = a + b + c \) and \( p(-1) = a - b + c \).  Wanting to
  write \( p^\prime(0) = d_1 p(1) + d_{-1} p(-1) \) for constants \(
  d_1 \) and \( d_{-1} \) leads to a \( 2 \times 2 \) linear system,
  which is readily solved to give \( d_1 = -d_{-1} = \tfrac{1}{2} \).
\end{proof}
Then, by a change of variables, this rule can be mapped to
\( \left[ -\tfrac{\ell}{2}, \tfrac{\ell}{2} \right] \) so that
\[
p^\prime(0) = \tfrac{1}{\ell} \left( p(\tfrac{\ell}{2}) - p(-\tfrac{\ell}{2}) \right).
\]
Finally, one can apply this rule on the edge of a triangle running
from \( \mathbf{v}_a \) to \( \mathbf{v}_b \) to find that
\[
\pi \delta^{\mathbf{t}_i} = \tfrac{\ell}{2}
\left( \pi \delta_{\mathbf{v}_b} - \pi \delta_{\mathbf{v}_a} \right).
\]

It is interesting to explicitly compute the product
\(V = E V^C D\), as giving a single formula rather than product of
matrices is more useful in practice.  Multiplying through gives:
\begin{equation}
  V =
  \begin{bmatrix}
    1 & 0 & 0 & 0 & 0 & 0 \\
    0 & 1 & 0 & 0 & 0 & 0 \\
    0 & 0 & 1 & 0 & 0 & 0 \\
    0 & \tfrac{-B^1_{12}}{\ell_1} & \tfrac{B^1_{12}}{\ell_1} & B^1_{11}
    & 0 & 0 \\
    \tfrac{-B^2_{12}}{\ell_2} & 0 & \tfrac{B^2_{12}}{\ell_2} & 0 &
    B^2_{11} & 0 \\
    \tfrac{-B^3_{12}}{\ell_3} & \tfrac{B^3_{12}}{\ell_3} & 0 & 0 & 0 & B^3_{11}
    \end{bmatrix}
\end{equation}

From the definition of $B^i$, it is possibly to explicitly calculate
its entries in terms of the those of the Jacobian and the normal and
tangent vectors for $K$ and $\hat{K}$.  Only the first row of each
$B^i$ is needed
\begin{equation}
  \begin{split}
    B^i_{11} & = \hat{n}^{\mathbf{x}}_i
    \left( n^{\mathbf{x}}_i \tfrac{\partial x}{\partial \hat{x}}
    + t^{\mathbf{x}}_i \tfrac{\partial y}{\partial \hat{x}} \right)
    + \hat{t}^{\mathbf{x}}_i
    \left( n^{\mathbf{x}}_i \tfrac{\partial x}{\partial \hat{y}}
    + t^{\mathbf{x}}_i \tfrac{\partial y}{\partial \hat{y}} \right) \\
    B^i_{12} & = \hat{n}^{\mathbf{x}}_i
    \left( n^{\mathbf{y}}_i \tfrac{\partial x}{\partial \hat{x}}
    + t^{\mathbf{y}}_i \tfrac{\partial y}{\partial \hat{x}} \right)
    + \hat{t}^{\mathbf{x}}_i
    \left( n^{\mathbf{y}}_i \tfrac{\partial x}{\partial \hat{y}}
    + t^{\mathbf{y}}_i \tfrac{\partial y}{\partial \hat{y}} \right) \\
  \end{split}
\end{equation}
We can also recall that the normal and tangent vectors are
related by $n^{\mathbf{x}} = t^{\mathbf{y}}$ and $n^{\mathbf{y}} = -t^{\mathbf{x}}$ to express these entries
purely in terms of either the normal or tangent vectors.  
Each entry of the Jacobian and normal and tangent vectors of $K$ and
$\hat{K}$ enter into the transformation.

In this form, \( V \) has 12 nonzero entries, although the
formation of those entries, which depend on normal and tangent vectors
and the Jacobian, from the vertex coordinates requires an
additional amount of arithmetic.  The Jacobian will typically be
computed anyway in a typical code,
and the cost of working with \( M =
V^T \) will again be subdominant to the nested loops over basis
functions and quadrature points required to form element matrices,
much like Hermite.

\subsubsection{The Argyris element}
Because it is higher degree than Morley and contains second
derivatives among the nodes, the Argyris transformation is more
involved.  However, it is a prime motivating example and also demonstrates that 
the general theory here reproduces the specific technique
in~\cite{dominguez2008algorithm}.
The classical Argyris element has $P$ as polynomials of degree 5 over
a triangle \( K \), a 21-dimensional space.
The 21 associated nodes $N$ are
selected as the point values, gradients, and all three unique second derivatives at
the vertices together with the normal derivatives evaluated at edge
midpoints.  These nodal choices lead to a proper $C^1$
element, and $C^2$ continuity is obtained at vertices.

Since the Argyris elements do not form an
affine-interpolation equivalent family, we will need to embed the
physical nodes into a larger set.  Much as with Morley elements, the
edge normal derivatives will be augmented by the tangential
derivatives.

With this notation, $N$ is a vector of 21 functionals and $N^C$ a
vector of 24 functions written as
\begin{equation}
  \begin{split}
  N & = \left[ \begin{array}{cccccccccccc}
    \delta_{\mathbf{v}_1} &
    \nabla_{\mathbf{v}_1} &
    \bigtriangleup_{\mathbf{v}_1} &
    \delta_{\mathbf{v}_2} &
    \nabla_{\mathbf{v}_2} &
    \bigtriangleup_{\mathbf{v}_2} &
    \delta_{\mathbf{v}_3} &
    \nabla_{\mathbf{v}_3} &
    \bigtriangleup_{\mathbf{v}_3} &
    \delta^{\mathbf{n}_1}_{\mathbf{e}_1} &
    \delta^{\mathbf{n}_2}_{\mathbf{e}_2} &
    \delta^{\mathbf{n}_3}_{\mathbf{e}_3}
  \end{array} \right]^T,
  \\
  N^C &=
  \left[\begin{array}{cccccccccccc}
    \delta_{\mathbf{v}_1} &
    \nabla_{\mathbf{v}_1} &
    \bigtriangleup_{\mathbf{v}_1} &
    \delta_{\mathbf{v}_2} &
    \nabla_{\mathbf{v}_2} &
    \bigtriangleup_{\mathbf{v}_2} &
    \delta_{\mathbf{v}_3} &
    \nabla_{\mathbf{v}_3} &
    \bigtriangleup_{\mathbf{v}_3} &
    \nabla^{\mathbf{n_1t_1}}_{\mathbf{v}_1} &
    \nabla^{\mathbf{n_2t_2}}_{\mathbf{v}_2} &
    \nabla^{\mathbf{n_3t_3}}_{\mathbf{v}_3}
    \end{array} \right]^T,
  \end{split}
\end{equation}
with corresponding ordering of reference nodes \( \hat{N} \) and \(
\hat{N}^c \).
The \(21 \times 24\) matrix $E$ just selects out the items in
$N^C$ that are also in $N$, so that
\[
E_{ij} = 
\begin{cases}
  1, & \text{for } 1 \leq i=j \leq 19 \ \text{or } (i,j) \in \left\{(20, 21), (21, 23)\right\} \\
  0, & \text{otherwise.}
\end{cases}
\]

The matrix $V^C$ relating the push-forward of the extended nodes to
the extended reference nodes is block diagonal and similar to our
earlier examples.  We use~\eqref{eq:justthechainrule} to map
the vertex gradient nodes as in the Hermite case.  Mapping the three unique
second derivatives by the chain rule requires the matrix:
\begin{equation}
  \Theta
  = \begin{bmatrix}
    \left( \tfrac{\partial \hat{x}}{\partial x} \right)^2
    & 2 \tfrac{\partial \hat{x}}{\partial x} \tfrac{\partial
      \hat{y}}{\partial x}
    & \left( \tfrac{\partial \hat{y}}{\partial x} \right)^2 \\
    \tfrac{\partial \hat{x}}{\partial y}
    \tfrac{\partial \hat{x}}{\partial x}
    &
    \tfrac{\partial{\hat{x}}}{\partial y}
    \tfrac{\partial\hat{y}}{\partial x}
    + \tfrac{\partial \hat{x}}{\partial x}
    \tfrac{\partial \hat{y}}{\partial y}
    & \tfrac{\partial \hat{y}}{\partial x} \tfrac{\partial
      \hat{y}}{\partial y} \\
    \left(\tfrac{\partial \hat{x}}{\partial y} \right)^2
    & 2 \tfrac{\partial \hat{x}}{\partial y} \tfrac{\partial
      \hat{y}}{\partial y}
    & \left( \tfrac{\partial \hat{y}}{\partial y} \right)^2
    \end{bmatrix}
\end{equation}
The edge midpoint nodes transform by $B$ just as
in~\eqref{eq:B}, so that the $V^C$ is
\begin{equation}
  \label{eq:VC}
  V^C =
  \left[
    \begin{array}{cccccccccccc}
    1 & 0 & 0 & 0 & 0 & 0 & 0 & 0 & 0 & 0 & 0 & 0 \\
    0 & J^{-T} & 0 & 0 & 0 & 0 & 0 & 0 & 0 & 0 & 0 & 0 \\
    0 & 0 & \Theta^{-1} & 0 & 0 & 0 & 0 & 0 & 0 & 0 & 0 & 0 \\
    0 & 0 & 0 & 1 & 0 & 0 & 0 & 0 & 0 & 0 & 0 & 0 \\
    0 & 0 & 0 & 0 & J^{-T} & 0 & 0 & 0 & 0 & 0 & 0 & 0\\
    0 & 0 & 0 & 0 & 0 & \Theta^{-1} & 0 & 0 & 0 & 0 & 0 & 0 \\
    0 & 0 & 0 & 0 & 0 & 0 & 1 & 0 & 0 & 0 & 0 & 0 \\
    0 & 0 & 0 & 0 & 0 & 0 & 0 & J^{-T} & 0 & 0 & 0 & 0 \\
    0 & 0 & 0 & 0 & 0 & 0 & 0 & 0 & \Theta^{-1} & 0 & 0 & 0 \\
    0 & 0 & 0 & 0 & 0 & 0 & 0 & 0 & 0 & B^1 & 0 & 0 \\
    0 & 0 & 0 & 0 & 0 & 0 & 0 & 0 & 0 & 0 & B^2 & 0 \\
    0 & 0 & 0 & 0 & 0 & 0 & 0 & 0 & 0 & 0 & 0 & B^3 \\
    \end{array}
    \right].
\end{equation}

Constructing $D$, like for Morley, is slightly more delicate.  The
additional nodes acting on quintic polynomials -- tangential
derivatives at edge midpoints -- must be written in terms of the
remaining nodes.  The first aspect of this involves a univariate
interpolation-theoretic question.  On the biunit interval $[-1, 1]$,
we seek a rule of the form
\[
f'(0) \approx a_1 f(-1) + a_2 f(1) + a_3 f^\prime(-1) + a_4
f^\prime(1) + a_5 f^{\prime\prime}(-1) + a_6 f^{\prime\prime}(1)
\]
that is exact when $f$ is a quintic polynomial.  The coefficients may
be determined to by writing a $6\times 6$ linear system asserting
correctness on the monomial basis.  The answer, given in~\cite{dominguez2008algorithm}, is
that
\begin{proposition}
  Any quintic polynomial \( p \) defined on \( [-1, 1] \) satisfies
\begin{equation}
p^\prime(0) = \tfrac{15}{16} \left( p(1) - p(-1) \right)
- \tfrac{7}{16} \left( p^\prime(1) + p^\prime(-1) \right)
+ \tfrac{1}{16} \left( p^{\prime\prime}(1) - p^{\prime\prime}(-1) \right).
\end{equation}
\end{proposition}
This can be mapped to the interval $[-\tfrac{\ell}{2}, \tfrac{\ell}{2}]$
by a change of variables:
\begin{equation}
  p^\prime(0) = \tfrac{15}{8\ell} \left( p\left(\tfrac{\ell}{2}\right)
  - p\left(\tfrac{-\ell}{2}\right) \right)
- \tfrac{7}{16} \left( p^\prime\left(\tfrac{\ell}{2}\right) + p^\prime\left(\tfrac{-\ell}{2}\right)\right)
+ \tfrac{\ell}{32} \left( p^{\prime\prime}\left(\tfrac{\ell}{2}\right) - p^{\prime\prime}\left(\tfrac{-\ell}{2}\right) \right).
\end{equation}

Now, we can use this to compute the tangential derivative at an edge
midpoint, expanding the tangential first and second derivatives in
terms of the Cartesian derivatives.  If $\mathbf{v}_a$ and
$\mathbf{v}_b$ are the beginning and ending vertex of edge \(\gamma_i\) with
midpoint $\mathbf{e}_i$ and length $\ell_i$, we write the tangential
derivative acting on quintics as
\begin{equation}
\begin{split}
\pi \delta^{\mathbf{t}_i}_{\mathbf{e}_i} = &
\tfrac{15}{8\ell_i}
\left(
\delta_{\mathbf{v}_b} - \delta_{\mathbf{v}_a}
\right)
- \tfrac{7}{16}
\left(
t^{\mathbf{x}}_i
\left( \delta^{\mathbf{x}}_{\mathbf{v}_b}
+ \delta^{\mathbf{x}}_{\mathbf{v}_a} \right)
+ 
t^{\mathbf{y}}_i
\left(
\delta^{\mathbf{y}}_{\mathbf{v}_b}
+ \delta^{\mathbf{y}}_{\mathbf{v}_a} \right)
\right)
\\
& + \tfrac{\ell_i}{32}
\left(
(t_i^{\mathbf{x}})^2 \left( \delta^{\mathbf{xx}}_{\mathbf{v}_b}
- \delta^{\mathbf{xx}}_{\mathbf{v}_a} \right)
+ 2 t_i^{\mathbf{x}} t_i^{\mathbf{y}} \left( \delta^{\mathbf{xy}}_{\mathbf{v}_b}
- \delta^{\mathbf{xy}}_{\mathbf{v}_a} \right)
+(t_i^{\mathbf{y}})^2 \left( \delta^{\mathbf{yy}}_{\mathbf{v}_b}
- \delta^{\mathbf{yy}}_{\mathbf{v}_a} \right)
\right).
\end{split}
\end{equation}

For each edge \( \gamma_i \), define the vector \( \mathbf{\tau}_i\) by
\[
\mathbf{\tau}_i = \begin{bmatrix} (t^{\mathbf{x}}_i)^2 & 2 t^{\mathbf{x}}_i t^{\mathbf{y}}_i & (t^{\mathbf{y}}_i)^2 \end{bmatrix}^T.
\]

The end result is that
\begin{equation}
  D = \left[
    \begin{array}{cccccccccccc}
      1 & 0 & 0 & 0 & 0 & 0 & 0 & 0 & 0 & 0 & 0 & 0 \\
      0 & I_2 & 0 & 0 & 0 & 0 & 0 & 0 & 0 & 0 & 0 & 0 \\
      0 & 0 & I_3 & 0 & 0 & 0 & 0 & 0 & 0 & 0 & 0 & 0 \\
      0 & 0 & 0 & 1 & 0 & 0 & 0 & 0 & 0 & 0 & 0 & 0 \\
      0 & 0 & 0 & 0 & I_2 & 0 & 0 & 0 & 0 & 0 & 0 & 0 \\
      0 & 0 & 0 & 0 & 0 & I_3 & 0 & 0 & 0 & 0 & 0 & 0 \\
      0 & 0 & 0 & 0 & 0 & 0 & 1 & 0 & 0 & 0 & 0 & 0 \\
      0 & 0 & 0 & 0 & 0 & 0 & 0 & I_2 & 0 & 0 & 0 & 0 \\
      0 & 0 & 0 & 0 & 0 & 0 & 0 & 0 & I_3 & 0 & 0 & 0 \\
      0 & 0 & 0 & 0 & 0 & 0 & 0 & 0 & 0 & 1 & 0 & 0 \\
      0 & 0 & 0 &
      \tfrac{-15}{8\ell_1} & \tfrac{7}{16} \mathbf{t}^T_1 &
      \tfrac{-\ell}{32} \mathbf{\tau}^T_1 & 
      \tfrac{15}{8\ell_1} & \tfrac{7}{16} \mathbf{t}^T_1 &
      \tfrac{\ell}{32} \mathbf{\tau}^T_1 & 
      0 & 0 & 0 \\
      0 & 0 & 0 & 0 & 0 & 0 & 0 & 0 & 0 & 0 & 1 & 0 \\
      \tfrac{-15}{8\ell_2} & \tfrac{7}{16} \mathbf{t}^T_2 &
      \tfrac{-\ell}{32} \mathbf{\tau}^T_2 &
      0 & 0 & 0 &      
      \tfrac{15}{8\ell_2} & \tfrac{7}{16} \mathbf{t}^T_2 &
      \tfrac{\ell}{32} \mathbf{\tau}^T_2 & 
      0 & 0 & 0 \\
      0 & 0 & 0 & 0 & 0 & 0 & 0 & 0 & 0 & 0 & 0 & 1 \\
      \tfrac{-15}{8\ell_3} & \tfrac{7}{16} \mathbf{t}^T_3 &
      \tfrac{-\ell}{32} \mathbf{\tau}^T_3 &
      \tfrac{15}{8\ell_3} & \tfrac{7}{16} \mathbf{t}^T_3 &
      \tfrac{\ell}{32} \mathbf{\tau}^T_3 & 
      0 & 0 & 0 &      
      0 & 0 & 0 \\
    \end{array}
    \right].
\end{equation}

If this transformation is kept in factored form, \( D \) contains 57
nonzero entries and \( V^c \) contains 54 nonzero entries.   \(E \) is
just a Boolean matrix and its application requires copies.  So,
application of \( M \) requires no more than \( 111 \) floating-point
operations, besides the cost of forming the entries themselves.  While
this is about ten times the cost of the Hermite transformation, it is 
for about twice the number of basis functions and still well-amortized
over the cost of integration loops.  Additionally, one can multiply
out the product \( E V^c D \) symbolically and find only 81 nonzero
entries, which reduces the cost of multiplication accordingly.

\subsection{Generalizations}
\subsubsection{Non-affine mappings}
Non-affine geometric transformations, whether for simplicial or other
element shapes, present no major complications
to the theory.  In this case, \( K \) and \( \hat{K} \) are related by
a non-affine map, and \( P \) is taken to be the image of \( \hat{P}
\) under pull-back
\begin{equation}
  P = \left\{ F^*(\hat{p}) : \hat{p} \in \hat{P} \right\},
\end{equation}
although this space need not consist of polynomials for non-affine \(
F \).  At any rate, one may define Hermite elements on curvilinear
cells~\citep{ciarlet1978finite,dautray2012mathematical}.  In this
case, the Jacobian matrix varies spatially so that each instance of \(
J^T \) in~\eqref{eq:blockguy} must be replaced by the particular value
of \( J^T \) at each vertex.  

\subsubsection{Generalized pullbacks}
Many vector-valued finite element spaces make use of pull-backs other
than composition with affine maps.  For example, the Raviart-Thomas
and N\'ed\'elec elements use contravariant and covariant Piola maps,
respectively.  Because these preserve either normal or tangential
components, one can put the nodal basis functions of a given element
\( (K, P, N) \) and reference element \( (\hat{K}, \hat{P}, \hat{N})
\) into one-to-one correspondence by means of the Piola transform,
a fact used heavily in~\citep{RognesKirbyEtAl2009a} possible.  It
would be straightforward to give a generalization of affine
equivalence to equivalence under an arbitrary pull-back \( F^* \),
with push-forward defined in terms of \( F^* \).  In this case, the
major structure of \S~\ref{sec:lagrange} would be unchanged.

However, not all \( H(\mathrm{div}) \) elements form equivalent
families under the contravariant Piola transform.  For example,
Mardal, Tai, and Winther~\cite{mardal2002robust} give an element
that can be paired with discontinuous polynomials
to give uniform inf-sup stability on a scale of spaces between \(
H(\mathrm{div}) \) and \( (H^1)^2 \), although it is \( H^1 \)-nonconforming.
The degrees of freedom include constant and linear moments of normal
components on edges, which are preserved under Piola mapping.
However, the nodes also include the constant moments of the tangential
component on edges, which are \emph{not} preserved under Piola
transform.  One could push-forward both the normal and tangential
constant moments, then express them as a linear combination of the
normal and tangential moments on the reference cell in a manner
like~\eqref{eq:blockguy}.  One could see the
Mardal--Tai--Winther element as satisfying a kind of
``Piola-interpolation equivalence'' and readily adapt the techniques
for Hermite elements,

\subsection{A further note on computation}
We have commented on the added cost of multiplying the set of basis
functions by \( M \) during local integration.  It also also possible
to apply the transformation in a different way that perhaps more fully
leverages pre-existing computer routines.  With this approach, \( M \)
can also be included in local matrix assembly by means
means of a congruence transform acting on the ``wrong'' element matrix
as follows.

Given a finite element \( (K, P, N) \) with nodal basis \( \Psi = \{
\psi_i \}_{i=1}^\nu \) and bilinear form
\( a_K(\cdot, \cdot) \) over the domain \( K \), we want to compute the matrix
\begin{equation}
  A^K_{ij} = a_K(\psi_j, \psi_i).
\end{equation}
Suppose that a computer routine existed for evaluating \( A^K \) via a
reference mapping for affine-equivalent elements.  That is, given the
mapping \( F:\hat{K} \rightarrow K \), this routine maps all
integration to the reference domain \( \hat{K} \) assuming that
the integrand over \( K \) is just the affine pull-back of something
on \( \hat{K} \).  Consider the following computation:

\begin{equation}
  \begin{split}
    A^K_{ij} & = a_K(\psi_j, \psi_i) \\
    & = a_K \left( \sum_{\ell_2=1}^\nu M_{j\ell_2} F^*(\hat{\psi}_{\ell_2}) ,
    \sum_{\ell_1=1}^\nu M_{i\ell_1} F^*(\hat{\psi}_{\ell_1}) \right) \\
    & = \sum_{\ell_1, \ell_2 = 1}^\nu M_{j\ell_2} M_{i\ell_1} a_K( F^*(\hat{\psi}_{\ell_2}),
    F^*(\hat{\psi}_{\ell_1})) \\
  \end{split}
\end{equation}
Now, this is just expressed in terms of the affine pullback of
reference-element integrands and so could use the hypothesized computer routine.
We then have
\begin{equation}
  A^K_{ij} = \sum_{\ell_1, \ell_2 = 1}^\nu
    M_{j\ell_2} M_{i\ell_1} a_{\hat{K}}( \hat{\psi}_{\ell_2},
    \hat{\psi}_{\ell_1})
    = \sum_{\ell_1, \ell_2 = 1}^\nu
    M_{j\ell_1} M_{i\ell_2} \hat{A}^K_{\ell_1\ell_2},
\end{equation}
or, more compactly,
\[
A^K = M \tilde{A}^K M^T,
\]
where \( \tilde{A}^K \) is the matrix one would obtain by using the
pull-back of the reference element nodal basis functions instead of
the actual nodal basis for \( (K, P, N) \).  Hence, rather than
applying \( M \) invasively at each quadrature point, one may use
existing code for local integration and pre- and post-multiply the
resulting matrix by the basis transformation.  In the case of Hermite,
for example, 
applying \( M \) to a vector costs 12 operations, so applying \( M \)
to all 10 columns of \( \tilde{A}^K \) costs 120 operations, plus
another 120 for the transpose.  This adds 240 extra operations to the
cost of building \( \tilde{A}^K \), or just 2.4 extra FLOPs per entry
of the matrix.

One may also apply this idea in a ``matrix-free'' context.  Given a
routine for applying \( \tilde{A}^K \) to a vector, one may simply
apply \( M^T \) to the input vector, apply \( \tilde{A}^K \) to the
result, and post-multiply by \( M \).  Hence, one has the cost of
muliplying by \( \tilde{A}^K \) plus the cost of applying \( M \) and
its transpose to a single vector.  In the case of Hermite, one has the
cost of computing the ``wrong'' local matrix-vector product via an
existing kernel plus 24 additional operations.

Finally, we comment on evaluating discrete functions over
elements requiring such transforms.  Discrete function evaluation is
frequently required in matrix-free computation, nonlinear residual
evaluation, and in bilinear form evaluation when a coefficient is
expressed in a finite element space.
Suppose one has on a local element \( K \) a function expressed by
\[
u = \sum_{j=1}^\nu c_j \psi_j,
\]
where \( c \in \mathbb{R}^\nu \) is the vector of coefficients and
\( \{ \psi_j \} \) is the nodal basis for \( (K, P, N) \).
In terms of pulled-back reference basis functions, \( u \) is given by
\[
u = \sum_{j=1}^\nu c_j \left( \sum_{k=1}^\nu M_{jk} F^*(\hat{\psi}_k)
\right)
= \sum_{j, k=1}^\nu M_{jk} c_j F^*(\hat{\psi}_k),
\]
which can also be written as
\begin{equation}
u = \sum_{k=1}^\nu (M^T c)_k F^*(\hat{\psi}_k)
= \sum_{k=1}^\nu (V c)_k F^*(\hat{\psi}_k).
\end{equation}
Just as one can build element matrices by means of the ``wrong''
basis functions and a patch-up operation, one can also evaluate functions
by transforming the coefficients and then using the standard
pullback of the reference basis functions.  Such observations may
make incorporating nonstandard element transformations into existing
code more practical.

\section{What if $P \neq F^*(\hat{P})$?}
\label{sec:whatthebell}
The theory so far has been predicated on $F^*$ providing an isomorphism
between the reference and physical function spaces.  In certain cases,
however, this fails.  Our main motivation here is to transform the
Bell element, a near-relative of the quintic Argyris
element.  In this case, one takes $P$ to be the subspace of $P_5$ that
has cubic normal derivatives on edges rather than the typical quartic values.
This reduction of $P$ by three dimensions is accompanied 
by removing the three edge normal derivatives
at midpoints from $N$.  In general, however, the pull-back
$F^*(\hat{P})$ does not coincide with $P$.  Instead of cubic 
normal derivatives on edges, $F^*(\hat{P})$ has reduced degree in
some other direction corresponding to the image of the normal
under affine mapping.  The theory developed earlier
can be extended somewhat to resolve this situation.

\subsection{General theory: extending the finite element}
Abstractly, one may view the Bell element or other spaces built by
constraint as the intersection of the null spaces of a collection of
functionals acting on some larger space as follows.  Let $(K, P, N)$ be a finite
element.  Suppose that $P \subset \tilde{P}$ and that
\(
\{ \lambda_i
\}_{i=1}^\kappa \subset \left( C^k_b \right)^\prime
\)
are linearly independent functionals that when
acting on \( \tilde{P} \) satisfy
\begin{equation}
  \label{eq:PfromPtilde}
  P = \cap_{i=1}^\kappa \mathrm{null}(\lambda_i).
\end{equation}

The following result is not difficult to prove:
\begin{proposition}
  Let \( (K, P, N) \) be a finite element with
  \( \cap_{i=1}^\kappa\mathrm{null}(\lambda_i) = P \subset \tilde{P}
  \) as per~\eqref{eq:PfromPtilde}.  Similarly, let
  Let \( (\hat{K}, \hat{P}, \hat{N})\) be a reference
  element with
  \( \cap_{i=1}^\kappa\mathrm{null}(\hat{\lambda}_i) = \hat{P} \subset
  \tilde{\hat{P}} \).  Suppose that
  \(\tilde{P} = F^*(\tilde{\hat{P}}).\)
  Then \(P = F^*(\hat{P})\) iff
  \begin{equation}
    \label{eq:spancondition}
  \mathrm{span}\{F_*(\lambda_i)\}_{i=1}^\kappa =
  \mathrm{span}\{\hat{\lambda}_i\}_{i=1}^\kappa.
  \end{equation}
\end{proposition}

In the case of the Bell element, the span condition~\eqref{eq:spancondition}
fails and so that the function space is not
preserved under affine mapping.  Consequently, the theory of the
previous section predicated on this preservation does not directly
apply.  Instead, we proceed by making the following observation.
\begin{proposition}
  \label{prop:extendedelement}
  Let $(K, P, N)$ be a finite element with $P \subset \tilde{P}$
  satisfying \( P = \cap_{i=1}^\kappa \mathrm{null}(\lambda_i)\) for
  linearly independent functionals \( \{\lambda_i\}_{i=1}^\kappa \).
  Define
  \[
  \tilde{N} = \begin{bmatrix} N \\ L \end{bmatrix}
  \]
  to include the nodes of $N$ together with
  $L = \begin{bmatrix} \lambda_1 & \lambda_2 & \dots & \lambda_\kappa \end{bmatrix}^T$.
  Then $(K, \tilde{P}, \tilde{N})$ is
  a finite element.
\end{proposition}
\begin{proof}
  Since we have a finite-dimensional function space, it remains to
  show that $\tilde{N}$ is linearly independent and hence spans
  $\tilde{P}^\prime$.  Consider a linear combination in
  $\tilde{P}^\prime$ 
  \[
  \sum_{i=1}^\nu c_i n_i + \sum_{i=1}^\kappa d_i \lambda_i = 0.
  \]
  Apply this linear combination to any $p \in P$ to find
  \[
  \sum_{i=1}^\nu c_i n_i(p) = 0
  \]
  since \( \lambda_i(p) = 0 \) for \( p \in P\).
  Because $(K, P, N)$ is a finite element, the $n_i$ are linearly
  independent in \( P^\prime \) so $c_i = 0$ for $1 \leq i \leq \nu$.
  Applying the same linear combination to any $ \in \tilde{P}
  \backslash P$ then gives that $d_i=0$ since the constraint
  functionals are also linearly independent.  
\end{proof}
Given a nodal basis \( (K, \tilde{P}, \tilde{N})\), it is easy to
obtain one for \( (K, P, N) \). 
\begin{proposition}
  \label{prop:extendedelementbasis}
  Let $(K, P, N)$, $\{ \lambda_i\}_{i=1}^\kappa$, and $(K, \tilde{P}, \tilde{N})$ be as in Proposition~\ref{prop:extendedelement}.  Order the nodes in $\tilde{N}$ by \( \tilde{N} = \begin{bmatrix} N \\ L \end{bmatrix} \) with $L_i = \lambda_i$ for $1 \leq i \leq \kappa$.  Let $\{\tilde{\psi}_{i}\}_{i=1}^{\nu + \kappa}$ be the nodal basis for $(K, \tilde{P}, \tilde{N})$.
  Then \( \{ \tilde{\psi}_i \}_{i=1}^\nu \) is the nodal basis for $(K, P, N)$.
\end{proposition}
\begin{proof}
  Clearly, \(n_i(\tilde{\psi}_j) = \delta_{ij}\) for
  \( 1 \leq i, j \leq \nu \)
  by the ordering of the nodes in \(\tilde{N}\).  Moreover,
  \(\{\tilde{\psi}_i\}_{i=1}^\nu \subset P\) because
  \(\lambda_i(\tilde{\psi}_j) = 0\) for each \(1 \leq i \leq \kappa\).
\end{proof}

\subsection{The Bell element}
So, we can obtain a nodal basis for the Bell element or others with
similarly constrained function spaces by mapping the
nodal basis for a slightly larger finite element and extracting a
subset of the basis functions.  Let 
\( (K, P, N) \) and \( (\hat{K}, \hat{P}, \hat{N})\) be the Bell
elements over \(K\) and reference cell \( \hat{K} \).

Recall that the Legendre polynomial of degree $n$ is orthogonal
to polynomials of degree $n-1$ or less.
Let $\mathcal{L}^n$ be the Legendre polynomial of degree $n$ mapped from the biunit interval to  
edge \( \gamma_i \) of $K$.   Define a functional
\begin{equation}
  \label{eq:elldef}
  \lambda_i(p) = \int_{\gamma_i} \mathcal{L}^4(s)
  \left( \mathbf{n}_i \cdot \nabla p \right)ds.
\end{equation}
For any \( p \in P_5(K) \), its normal derivative on edge \( i \) 
is cubic iff $\lambda_i(p) = 0$.  So, the constraint functionals
are given in
\(  L = \begin{bmatrix} \lambda_1 & \lambda_2 & \lambda_3 \end{bmatrix}^T
\)
and $\tilde{N} = \begin{bmatrix} N \\ L \end{bmatrix}$ as in
Proposition~\ref{prop:extendedelement}.  We define
\begin{equation}
  \label{eq:ellhatdef}
  \hat{\lambda}_i(p) = \int_{\hat{\gamma}_i} \mathcal{L}^4(s)
  \left( \hat{\mathbf{n}}_i \cdot \nabla p \right)ds
\end{equation}
and hence $(\hat{K}, \hat{P}, \hat{N})$ as well as \( \hat{L} \) and
\( \tilde{\hat{N}} \) in a similar way.

$P$ and $\hat{P}$ are the constrained spaces --
quintic polynomials with cubic normal derivatives on edges, while
$\tilde{P}$ and $\tilde{\hat{P}}$ are the spaces of full
quintic polynomials over $K$ and $\hat{K}$, respectively.  We must
construct a nodal basis for \( (\hat{K}, \tilde{\hat{P}},
\tilde{\hat{N}}) \), map it to a
nodal basis for \( (K, \tilde{P}, \tilde{N}) \) by the techniques in
Section~\ref{sec:theory}, and then take the subset of basis functions
corresponding to the Bell basis.

This is accomplished by specifying a compatible nodal extension
of \( \tilde{N} \) and \( \tilde{\hat{N}} \) by including the
edge moments of \emph{tangential} derivatives against \( \mathcal{L}^4
\)  with those of \(\tilde{N}\) and \(\tilde{\hat{N}} \).  We define
\begin{equation}
  \label{eq:lamprime}
  \begin{split}
  \lambda_i^\prime(p) & = \int_{\gamma_i} \mathcal{L}^4(s) \left(
  \mathbf{t}_i \cdot \nabla p \right) ds, \\
  \hat{\lambda}_i^\prime(p) & = \int_{\hat{\gamma}_i} \mathcal{L}^4(s) \left(
  \hat{\mathbf{t}}_i \cdot \nabla p \right) ds. \\
  \end{split}
\end{equation}

We must specify the \( E \), \( V^c \), and \( D \) matrices for this
extended set of finite element nodes.  We focus first on \( D \),
needing to compute each \(
\lambda_i^\prime \) in terms of the remaining functionals.
As with Morley and Argyris, we begin with univariate results.

The following is readily confirmed, for example, by noting the right-hand
side is a quintic polynomial and computing values and first and second
derivatives at $\pm 1$:
\begin{proposition}
  Let $p$ be any quintic polynomial on \([-1, 1]\).  Then
  \begin{equation}
    \label{eq:interp}
  \begin{split}
    16 p(x) & =  -\left( x-1 \right)^3 \left(
    p^{\prime\prime}(-1) \left( x+1 \right)^2 
    + p^\prime(-1) \left( x+1\right) \left( 3x + 5 \right)
    + p(-1) \left( 3 x^2 + 9x + 8 \right) \right) \\
    & + \left( x + 1 \right)^3 \left(
    p^{\prime\prime}(1) \left( x-1 \right)^2 
     - p^\prime(1)  \left( x-1 \right)\left( 3x - 5 \right)
    + p(1) \left( 3 x^2 - 9 x + 8 \right) \right).
  \end{split}
  \end{equation}
\end{proposition}

The formula~\eqref{eq:interp} can be differentiated and then integrated
against \( \mathcal{L}^4 \) to show that
  \begin{equation}
    \int_{-1}^1 p^\prime(x) \mathcal{L}^4(x) dx =
    \tfrac{1}{21}
    \left[p(1) - p(-1) - p^\prime(1) - p^\prime(-1) +
    \tfrac{1}{3} \left(p^{\prime\prime}(1) - p^{\prime\prime}(-1)\right) \right].
  \end{equation}

Then, this can be mapped to a general interval $[\tfrac{-\ell}{2},
  \tfrac{\ell}{2}]$ by a simple change of variables:
  \begin{equation}
    \int_{-\tfrac{\ell}{2}}^{\tfrac{\ell}{2}} p^\prime(x) \mathcal{L}^4(x) dx =
    \tfrac{1}{21}
    \left[p\left(\tfrac{\ell}{2}\right) - p\left(-\tfrac{\ell}{2}\right)
    - \tfrac{\ell}{2} \left(p^\prime\left(\tfrac{\ell}{2}\right)
    + p^\prime\left(-\tfrac{\ell}{2}\right)\right) +
    \tfrac{\ell^2}{12}\left( p^{\prime\prime}\left(\tfrac{\ell}{2}\right)
    - p^{\prime\prime}\left(-\tfrac{\ell}{2}\right)\right) \right].
  \end{equation}

Now, we can use this to express the functionals \( \lambda_i^\prime \) from~\eqref{eq:lamprime} as linear combinations of the Bell nodes:
\begin{proposition}
  Let $K$ be a triangle and $\mathbf{v}_a$ and $\mathbf{v}_b$ are the
  beginning and ending vertex of edge \( \gamma_i \) with
  length $\ell_i$.
  Let $p$ be any bivariate quintic polynomial over $K$ and $\lambda_i^\prime$
  defined in~\eqref{eq:lamprime}.  Then the restriction of $\lambda_i^\prime$
  to bivariate quintic polynomials satisfies 
  \begin{equation}
    \pi \lambda_i^\prime =
    \tfrac{1}{21}
    \left[
      \pi \delta_{\mathbf{v}_b} - \pi\delta_{\mathbf{v}_a}
      - \tfrac{\ell_i}{2}
      \left(\pi\delta_{\mathbf{v}_b}^{\mathbf{t}_i} 
      + \pi\delta_{\mathbf{v}_b}^{\mathbf{t}_i} 
      \right)
      + \tfrac{\ell_i^2}{12}
      \left(
      \pi\delta_{\mathbf{v}_b}^{\mathbf{t}_i\mathbf{t}_i}
      - \pi\delta_{\mathbf{v}_b}^{\mathbf{t}_i\mathbf{t}_i}
      \right)
      \right],
  \end{equation}
  and hence
  \begin{equation}
    \begin{split}
    \pi \lambda_i^\prime = & 
    \tfrac{1}{21}
    \left[
      \pi\delta_{\mathbf{v}_b} - \pi\delta_{\mathbf{v}_a}
      \right] \\
    & -\tfrac{\ell_i}{42}
      \left[
        t_i^{\mathbf{x}} \left( \pi\delta_{\mathbf{v}_b}^\mathbf{x}
      + \pi\delta_{\mathbf{v}_a}^\mathbf{x} \right)
      + t_i^{\mathbf{y}} \left( \pi\delta_{\mathbf{v}_b}^\mathbf{y}
      + \pi \delta_{\mathbf{v}_a}^\mathbf{y} \right)
      \right] \\
      & + 
      \tfrac{\ell_i^2}{252}
      \left(
      \left(t_i^{\mathbf{x}}\right)^2
      \left( \pi\delta_{\mathbf{v}_b}^{\mathbf{xx}}
      -\pi\delta_{\mathbf{v}_a}^{\mathbf{xx}}
      \right)
      +
      2 t_i^{\mathbf{x}}t_i^{\mathbf{y}}
      \left( \pi\delta_{\mathbf{v}_b}^{\mathbf{xy}} -\pi\delta_{\mathbf{v}_a}^{\mathbf{xy}}
      \right)
      + \left(t_i^{\mathbf{y}}\right)^2
      \left( \pi\delta_{\mathbf{v}_b}^{\mathbf{yy}} -\pi\delta_{\mathbf{v}_a}^{\mathbf{yy}}
      \right)
      \right).
    \end{split}
  \end{equation}
\end{proposition}

Now, \( V^c \) is quite similar to that
for the Argyris element.  There is a slight difference in the handling
the edge nodes, for we have an integral moment instead of a point
value and must account for the edge length accordingly.  By converting
between normal/tangent and Cartesian coordinates via the matrix \( G_i
\) and mapping to the reference element, we find that for any \( p \),
\begin{equation}
  \begin{split}
    \begin{bmatrix}
    \lambda_i(p) \\ \lambda_i^\prime(p)
    \end{bmatrix}
    & = \int_{\gamma_i} \mathcal{L}^4(s) \left( G_i \nabla p \right)
    ds \\
    & = \int_{\hat{\gamma}_i}
    \left|\tfrac{d\hat{s}}{ds} \right|\mathcal{L}^4(\hat{s}) \left( G_i J^T \hat{G}_i^T
    \hat{\nabla}^{\hat{\mathbf{n}}_i\hat{\mathbf{t}}_i} \hat{p} \right)
    d\hat{s}\\
    & = \left|\tfrac{d\hat{s}}{ds} \right|G_i J^T \hat{G}_i^T
    \begin{bmatrix} \hat{\lambda_i}(p) \\ \hat{\lambda_i}^\prime(p) \end{bmatrix}
  \end{split}
\end{equation}
This calculation shows that \( V^C \) for the Bell element is
identical to~\eqref{eq:VC} for Argyris, except with a geometric scaling
of the \( B \) matrices.

The extraction matrix \( E \) for the extended Bell elements
consisting of full quintics now is identical to that for Argyris.
Then, when evaluating basis functions, one multiplies the affinely
mapped set of basis values by \( V^T \) and then takes only the first
18 entries to obtain the local Bell basis.

\subsection{A remark on the Brezzi-Douglas-Fortin-Marini element}
In~\citep{Kir04}, we describe a two-part process for computing the
triangular Brezzi-Douglas-Fortin-Marini (BDFM)
element~\cite{fortin1991mixed}, an \(H(\mathrm{div})\) conforming finite element
based on polynomials of degree \( k \) with normal components
constrained to have degree \( k - 1 \).  This is a reduction of the
Brezzi-Douglas-Marini element~\cite{brezzi1985two} somewhat as
Bell is of Argyris.  However, as both elements form Piola-equivalent
families, the transformation techniques developed here are not needed.

Like the Bell element, one can define constraint functionals (integral
moments of normal components against the degree \( k \) Legendre
polynomial) for BDFM.  In~\citep{Kir04}, we formed a basis for the
intersection of the null spaces of these functionals by means of
a singular value  decomposition.  A nodal basis for the BDFM space
then followed by 
building and inverting a generalized Vandermonde matrix on the basis
for this constrained space.

In light of Propositions~\ref{prop:extendedelement}
and~\ref{prop:extendedelementbasis}, however, this process was rather
inefficient.  Instead, we could have merely extended the BDFM nodes by the
constraint functionals, building and inverting a single
Vandermonde-like matrix.  If one takes the BDM
edge degrees of freedom as moments of normal components against
Legendre polynomials up to degree $(k-1)$ instead of pointwise normal
values, then one can even build a basis for BDM that includes a
a basis for BDFM as a proper subset.

\section{Numerical results}
\label{sec:num}
Incorporation of these techniques into high-level software tools such
as Firedrake is the subject of ongoing investigation.  In the
meantime, we provide some basic examples written in Python, with
sparse matrix assemble and solvers using petsc4py~\citep{Dalcin:2011}.

\subsection{Scaling degrees of freedom}
Before considering the accuracy of the \( L^2 \)
projection, achieved via the global mass matrix,
we comment on the conditioning of the mass and other matrices
when both derivative
and point value degrees of freedom appear.  The Hermite element is
illustrative of the situation.

On a cell of typical diameter \( h \), consider a basis function
corresponding to the point value at a given vertex.  Since
the vertex basis function has a size of \( \mathcal{O}(1) \) on a triangle
of size \( \mathcal{O}(h^2) \), its \( L^2 \) norm should be \(
\mathcal{O}(h) \).  Now, consider a basis function corresponding to a
vertex derivative.  Its \emph{derivative} is now \( \mathcal{O}(1) \)
on the cell, so that the \( H^1 \) seminorm is \( \mathcal{O}(h) \).
Inverse inequalities suggest that the \( L^2 \) norm could then be as
large as \( \mathcal{O}(1) \).  That is, the different kinds of nodes
introduce multiple scales of basis function sizes under transformation, which
manifests in ill-conditioning.  Where one expects a mass matrix
to have an \( \mathcal{O}(1) \) condition number, one now obtains an
\( \mathcal{O}(h^{-2}) \) condition number.  This is observed even on
a unit square mesh, in Figure~\ref{fig:masscond}.  All condition numbers
are computed by converting the PETSc mass matrix to a dense matrix
and using LAPACK via scipy~\citep{scipy}

However, there is a simple solution.  For the Hermite element, one can
scale the derivative degrees of freedom locally by an ``effective \( h
\)''.  All cells sharing a given vertex must agree on that \( h \),
which could be the average cell diameter among cells sharing a
vertex.  Scaling the nodes/basis functions (which amounts to multiplying
\(V \) on the right by a diagonal matrix with 1's or \( h\)'s) removes
the scale separation among basis functions and leads again to an \(
\mathcal{O}(1) \) condition number for mass matrices, also seen in
Figure~\ref{fig:masscond}.  From here, we will assume that all degrees of freedom are
appropriately scaled to give \(\mathcal{O}(1) \) conditioning for the
mass matrix.

\begin{figure}
  \begin{center}
    \includegraphics[width=3.0in]{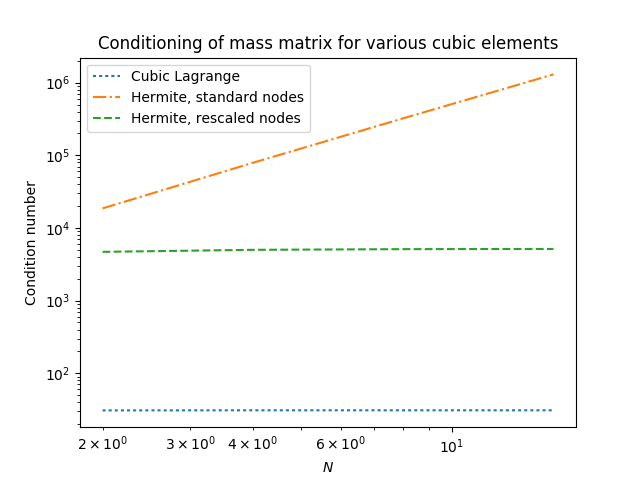}
  \end{center}
  \caption{Condition numbers for cubic Lagrange and Hermite mass
    matrices on an \( N \times N \) mesh divided into right triangles.
    This demonsrates an
    \( \mathcal{O}(h^{-2}) \) scaling when the ``original'' Hermite
    degrees of freedom are used, but \(
    \mathcal{O}(1) \) condition number when the 
    derivative degrees of freedom are scaled by \( h \).  
    Rescaling the Hermite nodes still gives a considerably larger
    condition number than for standard Lagrange elements.}
  \label{fig:masscond}
\end{figure}

\subsection{Accuracy of $L^2$ projection}
Now, we demonstrate that optimal-order accuracy is obtained by
performing \( L^2 \) projection of smooth functions into the Lagrange,
Hermite, Morley, Argyris, and Bell finite element spaces.  In each case we use an \(
N \times N \) mesh divided into right triangles.   Defining \( u(x, y)
= \sin(\pi x) \sin(2 \pi y) \) on \( [0,1]^2 \), we seek \( u_h \)
such that
\begin{equation}
  \left( u_h , v_h \right) = \left( u , v_h \right)
\end{equation}
for each \( v_h \in V_h \), where \( V_h \) is one of the the finite
element spaces.  Predicted asymptotic convergence rates -- third for Morley,
fourth for Hermite and Lagrange, fifth for Bell, and sixth for Argyris, are observed in Figure~\ref{fig:l2proj}.

Note that the Hermite and Lagrange elements have the same order of
approximation, but the Lagrange element delivers a slightly lower
error.  This is to be expected, as the space spanned by cubic Hermite
triangles is a proper subset of that spanned by Lagrange.

\begin{figure}
  \begin{center}
    \includegraphics[width=3.0in]{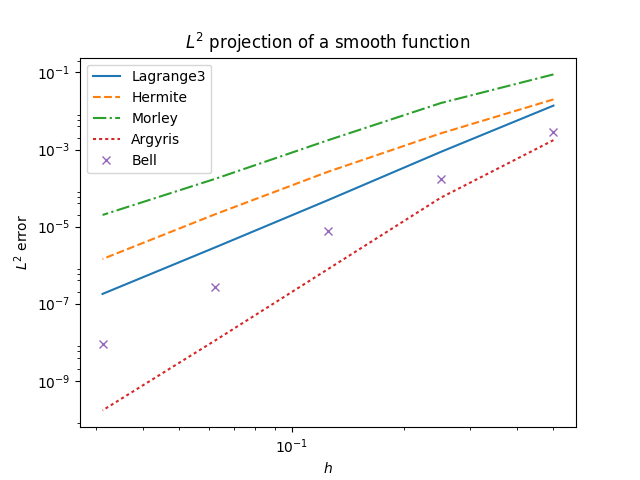}
  \end{center}
  \caption{Accuracy of \( L^2 \) projection using cubic Lagrnage,
    Hermite, Morley, Argyris, and Bell elements.  All approach theoretically optimal rates.}
  \label{fig:l2proj}
\end{figure}
  
\subsection{The Laplace operator}
As a simple second-order elliptic operator, we consider the Dirichlet
problem for the Laplace
operator on the unit square \( \Omega\):
\begin{equation}
  \label{eq:laplace}
  -\Delta u = f,
\end{equation}
equipped with homogeneous Dirichlet boundary conditions
\( u = 0 \) on \( \partial \Omega \).

We divide \( \Omega \) into an \( N \times N \) mesh of triangles and
let \( V_h \) be one of the Lagrange, Hermite, Argyris, or Bell finite
element spaces, all of which are \( H^1 \)-conforming,
over this mesh.  The Morley element is not a suitable \( H^1 \)
nonconforming element, so we do not use it here.
We then seek \(
u_h \in V_h \) such that
\begin{equation}
  \left( \nabla u_h , \nabla v_h \right)
  = \left( f , v_h \right)
\end{equation}
for all \( v_h \in V_h \).

Enforcing strong boundary conditions on elements with derivative
degrees of freedom is delicate in general.  However, with grid-aligned
boundaries, it is less difficult.  To force a function to be zero on a
given boundary segment, we simply require the vertex values and all
derivatives tangent to the edge vanish.  This amounts to setting the
\(x\)-derivatives on the top and bottom edges of the box and
\(y\)-derivative on the left and right for Hermite, Argyris, and Bell
elements.  Dirichlet conditions for Lagrange are enforced in the
standard way.

By the method of manufactured solutions, we select \(
f(x, y) = 8 \pi^2 \sin(2 \pi x) \sin(2 \pi y) \) so that
\( u(x,y) = \sin(2 \pi x) \sin(2 \pi y) \).  In Figure~\ref{fig:laplaceerror}, we show
the \( L^2 \) error in the computed solution for both element
families.  As the mesh is refined, both curves approach the expected
order of convergence -- fourth for Hermite and Lagrange, fifth for
Bell, and sixth for Argyris.  Again, the error for Lagrange is
slightly smaller than for Hermite, albeit with more global degrees of freedom.

\begin{figure}
  \begin{center}
    \includegraphics[width=3.0in]{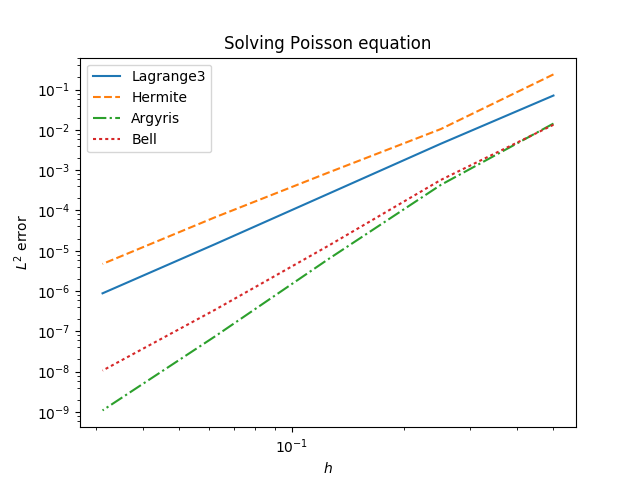}
  \end{center}
  \caption{Convergence study of various elements for
    second-order elliptic equation~\eqref{eq:laplace}.  As the mesh
    is refined, all elements approach their predicted optimal rates
    of convergence.}
  \label{fig:laplaceerror}
\end{figure}

\subsection{The clamped plate problem}
We now turn to a fourth-order problem for which the Argyris and Bell elements
provide conforming \( H^2 \) discretizations and Morley a suitable
nonconforming one. Following~\citep{BreSco}, we take the bilinear form
defined on \( H^2( \Omega ) \) to be
\begin{equation}
  a(u, v) = \int_{\Omega} \Delta u \Delta v
  - \left( 1 - \nu \right)
  \left( 2 u_{xx} v_{yy} + 2 u_{yy} v_{xx} - 4 u_{xy} v_{xy} \right)
  dx dy,
\end{equation}
where \( 0 < \nu < 1 \) yields a coercive bilinear form for any closed
subspace of \( H^2 \) that does not contain nontrivial linear polynomials.  We
fix \( \nu = 0.5 \).

Then, we consider the variational problem
\begin{equation}
  a(u, v) = F(v) = \int_\Omega f v \ dx,
  \label{eq:plate}
\end{equation}
posed over suitable subspaces of \( H^2 \).  It is known~\citep{BreSco} that
solutions of~\eqref{eq:plate} that lie in \(H^4(\Omega) \) satisfy the
biharmonic equation \( \Delta^2 u = f \) in an \( L^2 \) sense.

We consider the clamped plate problem, in which both the function
value and outward normal derivative are set to vanish, which removes
nontrivial linear polynomials from the space.  Again, we use
the method of manufactured solutions on the unit square to select \(
f(x,y) \) such that \( u(x,y) = \left( x(1-x) y(1-y) \right)^2 \),
which satifies clamped boundary conditions.  We solve this
problem with Argyris and Bell elements, and then also use the
nonconforming Morley element in the bilinear form.  Again, expected
orders of convergence are observed in Figure~\ref{fig:plateerr}.

\begin{figure}
  \begin{center}
    \includegraphics[width=3.0in]{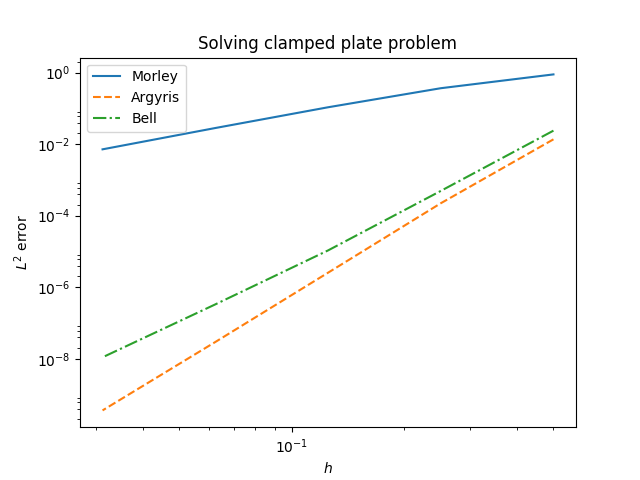}
  \end{center}
  \caption{Convergence study of Hermite and Argyris elements for
    clamped plate biharmonic problem~\eqref{eq:plate}.  As the mesh
    is refined, Bell and Argyris elements converge in \( L^2 \) at
    fifth and sixth order, respectively.  The nonconforming Morley
    element only converges at second order, which is known to be sharp.}
  \label{fig:plateerr}
\end{figure}

\section{Conclusions}
Many users have wondered why FEniCS, Firedrake, and most other high-level finite
element tools lack the full array of triangular elements, including
Argyris and Hermite.  One answer is that fundamental
mathematical aspects of mapping such elements have remained relatively
poorly understood.  This work demonstrates the challenges involved
with mapping such elements from a reference cell, but also proposes a 
general paradigm for overcoming those challenges by embedding the
nodes into a larger set that transforms more cleanly and using
interpolation techniques to relate the additional nodes back to
original ones.
In the future, we hope to incorporate these techniques
in FInAT (\url{https://github.com/FInAT/FInAT}), a successor project
to FIAT that produces abstract syntax for finite element evaluation
rather than flat tables of numerical values.  TSFC~\cite{tsfc} already
relies on FInAT to enable sum-factorization of
tensor-product bases.  If FInAT can provide rules for evaluating the
matrix \( M \) in terms of local geometry on a per-finite element basis,
then TSFC and other form compilers should be able to seamlessly (from the
end-users' perspective) generate code for many new kinds of finite
elements.

\bibliographystyle{plain}
\bibliography{bib}

\end{document}